\documentclass[11pt]{article}

\usepackage[margin=0.5in]{geometry}

\usepackage[english]{babel}
\usepackage{amsmath,amssymb,amsfonts,amsthm,enumerate}

\numberwithin{equation}{section}

\newtheorem{theorem}{Theorem}[section]
\newtheorem{lemma}[theorem]{Lemma}

\newtheorem{corollary}[theorem]{Corollary}

\usepackage{mathrsfs}
\renewcommand{\mathcal}{\mathscr}

\newcommand{\N}{\mathbb{N}}
\newcommand{\R}{\mathbb{R}}

\renewcommand{\H}{\mathbb{H}}
\newcommand{\B}{\mathcal{B}}
\newcommand{\D}{\mathcal{D}}
\newcommand{\E}{\mathcal{E}}

\renewcommand{\epsilon}{\varepsilon}
\newcommand{\eps}{\varepsilon}

\newcommand{\Per}{{\rm{Per}}}
\newcommand{\dist}{{\rm{dist}}}

\renewcommand{\leq}{\leqslant}
\renewcommand{\le}{\leqslant}

\renewcommand{\ge}{\geqslant}

\date{}

\title{Continuity and density results\\
for a one-phase nonlocal free boundary problem\thanks{It
is a great pleasure to thank
Filippo Cagnetti and Ovidiu Savin for several pleasant
and stimulating conversations.
This work has been supported by EPSRC grant EP/K024566/1
\emph{Monotonicity formula methods for nonlinear PDEs},
ERPem ``PECRE Postdoctoral and Early Career Researcher Exchanges'', 
Humboldt Foundation, ERC grant 277749 \emph{EPSILON Elliptic
Pde's and Symmetry of Interfaces and Layers for Odd Nonlinearities}
and PRIN grant 2012 \emph{Critical Point Theory
and Perturbative Methods for Nonlinear Differential Equations}.}} 

\author{Serena Dipierro\thanks{School of Mathematics and Statistics,
University of Melbourne
813 Swanston Street, Parkville VIC 3010, Australia. Email: 
{\tt{sdipierro@unimelb.edu.au}}.}
and Enrico Valdinoci\thanks{School of Mathematics and Statistics,
University of Melbourne
813 Swanston Street, Parkville VIC 3010, Australia, and
Weierstra{\ss} Institut
f\"ur Angewandte Analysis und 
Stochastik, Mohrenstra{\ss}e 39, 10117 Berlin, Germany, 
and Dipartimento di Matematica Federigo Enriques,
Universit\`a degli Studi di Milano,
Via Cesare Saldini 50, 20133 Milano, Italy, and 
Istituto di Matematica Applicata e Tecnologie Informatiche Enrico Magenes,
Consiglio Nazionale delle Ricerche,
Via Ferrata 1, 27100 Pavia, Italy.
Email: {\tt{enrico@math.utexas.edu}}.
}}

\begin{document}

\maketitle

\begin{abstract}
We consider a one-phase nonlocal free boundary problem
obtained by the superposition of a fractional Dirichlet energy
plus a nonlocal perimeter functional. We prove that the minimizers
are H\"older continuous and the free boundary has positive density
from both sides.

For this, we also introduce a new notion of fractional harmonic
replacement in the extended variables and we study its basic
properties.
\end{abstract}

\medskip

{\footnotesize
\smallskip

}\medskip

\tableofcontents

\section{Introduction}\label{INT}

The recent research has payed a great attention to a class
of nonlocal problems arising in both pure and applied mathematics.
A natural setting in which nonlocal questions arise is given
by the class of free boundary problems. Roughly speaking,
many free boundary problems are built by the competition of two (or
more) competing terms: for instance, an elastic (or ferromagnetic)
energy can be combined with a tension effect (in this setting,
the ferromagnetic energy favors the preservation of
the values of a state parameter~$u$, while the tension effect
tends to make the interface given by the level sets of~$u$
as small as possible).

In order to take into account possible long-range interactions,
some nonlocal energies have been considered in these types
of free boundary problems.
In particular, in~\cite{DSV2}, 
a new energy functional was considered, as the sum of
a fractional Dirichlet energy, with fractional exponent~$s\in(0,1)$,
and a fractional perimeter, with fractional exponent~$\sigma\in(0,1)$.
When~$s\to1$, and when~$\sigma\to0$ or~$\sigma\to1$,
the energy functional becomes the classical free boundary
energy considered in~\cite{AC, ACF, ACKS}. An intermediate problem,
with a local Dirichlet energy plus a fractional perimeter
has been studied in~\cite{CSV}.

Some results of classical flavor have been proved in~\cite{DSV2},
such as, among the others,\footnote{We take this opportunity to amend some minor
inconsistencies in~\cite{DSV2}.

First of all, in Theorem~1.2 and in Lemma~8.3, the condition
``${\overline{u}}_m$ is the extension of~$u_m$''
has to be intended ``${\overline{u}}_m$ is the extension of~$u_m
\in C(\R^n)$''.

Then, in the statement of Lemma~3.2 ``if~$u\in C(\R^n)$'' has to
be placed in the beginning, and in the proof of Lemma~3.2
the expression~``$\displaystyle\min_{B_r(x_o)} u$''
needs to be replaced by~``$\displaystyle\min_{\overline{B_r(x_o)}} u$''.}
a monotonicity formula for the minimizers, some glueing lemmata,
some uniform energy bounds, convergence results,
a regularity theory for the planar cones
and a trivialization result for the flat case.
On the other hand, in~\cite{DSV2} no result was proved
concerning the regularity of the minimizers and the density properties
of the free boundary. These type of results are indeed
quite hard to obtain, due to the strong nonlocal feature of
the problem: for instance, differently from the classical case,
the nonlocal Dirichlet energy provides nontrivial interactions
between the positivity and negativity sets of the functions,
and a local modification
of the free boundary produces global consequences in
the fractional perimeter.

Goal of this paper is then to provide regularity and density
results, at least in the case of the one-phase problem
(i.e. when the boundary data are nonnegative).

The mathematical setting in which we work is the following.
Let~$s$, $\sigma\in(0,1)$, 
and~$\Omega\subset\R^n$ be a bounded domain with Lipschitz boundary.
Following~\cite{DSV2}, we define
$$ {\mathcal{F}}_\Omega(u,E):=
\iint_{Q_\Omega}\frac{|u(x)-u(y)|^2}{|x-y|^{n+2s}}\,dx\,dy+
\Per_\sigma(E,\Omega),$$
where
$$ Q_\Omega:=(\Omega\times\Omega)\cup
\big( (\R^n\setminus\Omega)\times\Omega \big)\cup
\big( \Omega\times(\R^n\setminus\Omega)\big)$$
and~$\Per_\sigma(E,\Omega)$ denotes the fractional perimeter
of~$E$ in~$\Omega$ (see~\cite{CRS}
or formulas (1.2) and~(1.3) in~\cite{DSV2}), that is
\begin{equation}\label{oiU67jsdfgH} 
\Per_\sigma(E,\Omega):=L(E\cap \Omega, \Omega\setminus E)
+ L(E\cap \Omega, (\R^n\setminus E)\setminus\Omega)+
L (E\setminus\Omega, \Omega\setminus E),\end{equation}
where, for any disjoint sets~$A$, $B\subseteq\R^n$,
$$ L(A,B):=\iint_{A\times B}\frac{dx\,dy}{|x-y|^{n+\sigma}}.$$
All sets and
functions are implicitly assumed to be measurable from now on.

Let~$E\subseteq\R^n$ and~$u:\R^n\to\R$.
We say that~$(u,E)$ is an admissible pair if~$u\ge0$ a.e. in~$E$
and~$u\le0$ a.e. in~$\R^n\setminus E$. 

Also, we say that~$(u,E)$ is
a minimizing pair in~$\Omega$
if~${\mathcal{F}}_\Omega(u,E)<+\infty$
and
$$ {\mathcal{F}}_\Omega(u,E)\le
{\mathcal{F}}_\Omega(v,F)$$
for any admissible pair~$(v,F)$ such that:
\begin{itemize}
\item $u-v\in H^s(\R^n)$, 
\item $u=v$
a.e. in~$\R^n\setminus\Omega$, and
\item $E\setminus\Omega=F\setminus\Omega$.
\end{itemize}
Roughly speaking, a pair~$(u,E)$ is admissible if~$E$ is the positivity
set of~$u$, and it is minimizing if it has minimal energy
among all the possible competing admissible pairs
that coincide outside~$\Omega$. For the existence of minimizing pairs
see Lemma~3.1 in~\cite{DSV2}.

We remark that this minimizing problem is nontrivial even
in the one-phase case, i.e. when the boundary datum~$u$ is nonnegative,
since the set~$E$ is not necessarily trivially prescribed outside~$\Omega$.

In this setting, our main result is the following:

\begin{theorem}[Density estimates and continuity for
one-phase minimizers]\label{DEC}
Assume that~$(u,E)$ is minimizing in~$B_1$, with~$u\ge0$
a.e. in~$\R^n\setminus B_1$ and~$0\in \partial E$.
Assume also that
\begin{equation}\label{GROW}
\int_{\R^n} \frac{|u(x)|}{1+|x|^{n+2s}}\,dx\le \Lambda,\end{equation}
for some~$\Lambda>0$.

Then, there exist~$c$, $K>0$,
possibly depending on~$n$, $s$, $\sigma$ and~$\Lambda$, such that
for any~$r\in(0,1/2]$,
\begin{equation}\label{CD-2}
\min\Big\{ |B_r\cap E|,\, |B_r\setminus E|\Big\}\ge cr^n
\end{equation}
and
\begin{equation}\label{CD-2-BIS}
\|u\|_{L^\infty(B_{1/2})}\le K.
\end{equation}
In addition, if~$s>\sigma/2$, then, given~$r_0\in(0,1/4)$,
\begin{equation}\label{CD-1}
{\mbox{$u\in C^{s-\frac\sigma2}(B_{r_0})$, with~$\|
u\|_{C^{s-\frac\sigma2}(B_{r_0})}\le C$,}}
\end{equation}
where~$C>0$ possibly depends
on~$n$, $s$, $\sigma$, $r_0$ and~$\Lambda$.
\end{theorem}

We observe that both the growth condition~\eqref{GROW}
and the H\"older exponent in~\eqref{CD-1} are compatible
with the degree of homogeneity of the minimizing cones,
see Theorem~1.3 of~\cite{DSV2}. 
Condition~\eqref{GROW} is also a standard assumption
to make sense of the fractional Laplace operator
(though some very recent developments in~\cite{POLLI}
may also allow more general notions of suitable fractional Laplace operators
for functions with more severe growth at infinity).

It is an open problem
to investigate the optimal regularity of the solution
(which could be possibly beyond the scaling arguments) and
to classify (or trivialize) the minimizing cones:
see also~\cite{CSV, DSV2} for partial results and additional comments
on these problems.

It is also an interesting question
to study this type of
free boundary problems
for more general fractional operators
(see e.g.~\cite{ADD1, ADD2} for a classical
counterpart).
\bigskip

The rest of the paper is organized as follows.
In Section~\ref{EXT:S:1} we introduce an extension
problem which is useful to localize the Dirichlet
energy (using a weighted space with an additional variable).
This extended problem is different than the one considered in~\cite{DSV2}
since here the fractional perimeter functional is
not modified by the extension procedure.

In Section~\ref{FHR}, we introduce a
fractional harmonic replacement in this weighed extended space.
Fractional harmonic replacements are of course a classical
topic in harmonic analysis
and they have several applications to free boundary problems,
see e.g.~\cite{ACKS, CSV} and the references therein.
In the literature, a fractional harmonic replacement
was also studied in~\cite{DV}. The setting of~\cite{DV}
is different than the one considered in Section~\ref{FHR} of
this paper, since here we deal with the extended space
and, in Section~\ref{9uTHJkl789},
we obtain localized energy estimates in the extended variable.
These energy estimates
play a crucial role in our subsequent density estimates
(as a matter of fact, both the replacement of~\cite{DV}
and the one of Section~\ref{FHR} here will be used
in this paper to prove density estimates from both sides).

In Section~\ref{8ikY89521} we prove the
density estimates. First we prove the density
of the vanishing set around free boundary points, together
with a uniform estimate on the size of the solution.
Then we use this information to obtain density estimates
of the positivity set as well, which completes the proof
of the double-sided density estimate in~\eqref{CD-2}.

By combining the density estimates with the uniform bound
on the solution, one also obtains continuity of the minimizers,
as claimed in~\eqref{CD-1}.

\section{An extended problem}\label{EXT:S:1}

In this section we introduce an extension problem in order to localize 
the Dirichlet energy, by adding one variable (see \cite{CaffSilv}).  

We use the following setting.
We consider variables~$x\in\R^n$ and~$z\in\R$,
and we use the notation~$X:=(x,z)\in\R^{n+1}$.
We consider the halfspace~$\R^{n+1}_+ :=\R^n\times(0,+\infty)$.
The $n$-dimensional
ball centered at~$0\in\R^n$ and of radius~$r>0$ is denoted by~$B_r$.

Given~$u:\R^n\to\R$, for any~$(x,z)\in\R^{n+1}$
we define
\begin{equation}\label{EXTENDED}
\overline u(x,z):=
\int_{\R^n} \frac{|z|^{2s}\; u(x-y)}{(|y|^2+z^2)^{\frac{n+2s}{2}}}\,dy=
\int_{\R^n} \frac{|z|^{2s}\; u(y)}{(|x-y|^2+z^2)^{\frac{n+2s}{2}}}\,dy,
\end{equation}
see e.g. \cite{CaffSilv}, in particular Section $2.4$ there
(notice that in \cite{CaffSilv} in the definition of the extension function $\overline u$ 
there is also a normalizing constant, that we neglect here, since it will not play any role in our 
problem). 

Next result states that if~$(u,E)$ is a minimal pair,
then~$(\overline u,E)$ is minimal for an extended problem:

\begin{lemma}\label{EXT-L}
Let~$(u,E)$ be a minimizing pair in~$B_r$.
Let~${\mathcal{U}}$ be
a bounded and Lipschitz domain of~$\R^{n+1}$
that is symmetric with respect to the $z$-coordinate, such that
\begin{equation*}
{\mathcal{U}}\cap \{z=0\}\subset B_r\times\{0\}.\end{equation*}
Then
$$ \int_{\mathcal{U}} |z|^a |\nabla \overline u|^2\,dX
+\Per_\sigma(E,B_r)\le
\int_{\mathcal{U}} |z|^a |\nabla \tilde v|^2\,dX
+\Per_\sigma(F,B_r),$$
for every~$(\tilde v,F)$ such that:
\begin{itemize}
\item $F\setminus B_r=E\setminus B_r$,
\item $\tilde v-\overline u$ is compactly supported inside~${\mathcal{U}}$,
\item $\tilde v(x,0)\ge0$ a.e.~$x\in F$,
\item $\tilde v(x,0)\le0$ a.e.~$x\in\R^n\setminus F$.
\end{itemize}
\end{lemma}

\begin{proof} We take~$(u,E)$ and~$(\tilde v,F)$
as in the statement of Lemma~\ref{EXT-L}
and we define~$v(x):=\tilde v(x,0)$, for any~$x\in\R^n$.

Notice that~$(\R^n\setminus B_r)\times\{0\}\subseteq
\R^{n+1}\setminus{\mathcal{U}}$, therefore~$v(x)=\tilde v(x,0)=
\overline u(x,0)=u(x)$ for a.e.~$x\in\R^n\setminus B_r$.
In addition, $v\ge0$ a.e. on~$F$ and~$v\le0$ a.e. on~$\R^n\setminus F$.
Therefore, the pair~$(v,F)$ is an admissible competitor for~$(u,E)$
and so, by the minimality of~$(u,E)$, we have that
\begin{equation}\label{9gf567gfd789}
\begin{split}& \iint_{Q_{B_r}}\frac{|v(x)-v(y)|^2}{|x-y|^{n+2s}}\,dx\,dy
- \iint_{Q_{B_r}}\frac{|u(x)-u(y)|^2}{|x-y|^{n+2s}}\,dx\,dy \\
&\qquad= {\mathcal{F}}_{B_r}(v,F)
-{\mathcal{F}}_{B_r}(u,E)
-\Per_\sigma(F,B_r)+\Per_\sigma(E,B_r)
\\ &\qquad\ge
-\Per_\sigma(F,B_r)+\Per_\sigma(E,B_r).\end{split}\end{equation}
On the other hand, by Lemma~7.2 of~\cite{CRS}, up to a normalizing constant, we have that
\begin{eqnarray*}
&& \iint_{Q_{B_r}}\frac{|v(x)-v(y)|^2}{|x-y|^{n+2s}}\,dx\,dy
- \iint_{Q_{B_r}}\frac{|u(x)-u(y)|^2}{|x-y|^{n+2s}}\,dx\,dy
\\&&\qquad=\inf \int_{
{\mathcal{W}} } |z|^a 
\big( |\nabla \tilde w|^2 -|\nabla \overline u|^2\big)\,dX,
\end{eqnarray*}
where the infimum above is taken over all the couples~$(\tilde w,{\mathcal{W}})$
satisfying the following properties:
\begin{itemize}
\item ${\mathcal{W}}$ is
a bounded and Lipschitz domain of~$\R^{n+1}$
that is symmetric with respect to the $z$-coordinate, such that
$$ {\mathcal{W}}\cap \{z=0\}\subset B_r\times\{0\},$$
\item $\tilde w-\overline u$ is compactly supported inside~${\mathcal{W}}$,
\item $\tilde w(x,0)=v(x)$ for any~$x\in\R^n$.
\end{itemize}
By construction, we can take~$\tilde w:=\tilde v$ and~$
{\mathcal{W}}:={\mathcal{U}}$ as candidates in the above infimum,
and consequently
\begin{eqnarray*}
&& \iint_{Q_{B_r}}\frac{|v(x)-v(y)|^2}{|x-y|^{n+2s}}\,dx\,dy
- \iint_{Q_{B_r}}\frac{|u(x)-u(y)|^2}{|x-y|^{n+2s}}\,dx\,dy
\\&&\qquad\le\int_{
{\mathcal{U}} } |z|^a
\big( |\nabla \tilde v|^2 -|\nabla \overline u|^2\big)\,dX.\end{eqnarray*}
This and~\eqref{9gf567gfd789} give that
$$ \int_{
{\mathcal{U}} } |z|^a
\big( |\nabla \tilde v|^2 -|\nabla \overline u|^2\big)\,dX\ge
-\Per_\sigma(F,B_r)+\Per_\sigma(E,B_r),$$
that is the desired result.
\end{proof}

\section{Fractional harmonic replacements in the extended variables}\label{FHR}

Goal of this section is to introduce a notion of fractional harmonic
replacement in the extended variables and study its basic properties.
In the classical case, a detailed study of the harmonic
replacement was performed in~\cite{ACKS, CSV}.
See also~\cite{DV} for the study of a related (but different)
fractional harmonic replacement.

We set
\begin{equation}\label{90GHJ}
\B_r:= B_{\frac{9r}{10}} \times (-r,r).\end{equation}
It worth to link the norm in~$\B_r$ for the extended function
with the one on the trace, as pointed out by the following result:

\begin{lemma}\label{Cr Lemma}
Let~$u$ and~$\overline u$ be as in~\eqref{EXTENDED}.
There exists~$C_r>0$ such that
$$ \| \overline u\|_{L^\infty(\B_r)}\le
C_r\,\left( \| u\|_{L^\infty(B_r)} +
\int_{\R^n\setminus B_r} \frac{|u(y)|}{|y|^{n+2s}}\,dy
\right).$$
\end{lemma}

\begin{proof} Let~$(x,z)\in\B_r$.
Then~$x\in B_{\frac{9r}{10}}$ and~$|z|\le r$.
Therefore, if~$y\in \R^n\setminus B_r$, we have that
$$ |x-y|\ge |y|-|x| = \frac{|y|}{10} +\frac{9|y|}{10} -|x|
\ge \frac{|y|}{10} +\frac{9r}{10}-\frac{9r}{10} =\frac{|y|}{10}.$$
Hence, if~$y\in \R^n\setminus B_r$,
$$ \frac{|z|^{2s}\; |u(y)|}{(|x-y|^2+z^2)^{\frac{n+2s}{2}}}
\le \frac{r^{2s}\; |u(y)|}{|x-y|^{{n+2s}}} 
\le C_r \,\frac{|u(y)|}{|y|^{{n+2s}}},$$
for some~$C_r>0$. As a consequence
\begin{equation}\label{89yjkiii098}
\int_{\R^n\setminus B_r}
\frac{|z|^{2s}\; |u(y)|}{(|x-y|^2+z^2)^{\frac{n+2s}{2}}}\,dy
\le C_r \int_{\R^n\setminus B_r} \frac{|u(y)|}{|y|^{{n+2s}}}\,dy.\end{equation}
Moreover,
\begin{eqnarray*}&& \int_{B_r}
\frac{|z|^{2s}\; |u(y)|}{(|x-y|^2+z^2)^{\frac{n+2s}{2}}}\,dy
\le \|u\|_{L^\infty(B_r)} \int_{B_r}
\frac{|z|^{2s}\;}{(|x-y|^2+z^2)^{\frac{n+2s}{2}}}\,dy\\&&\qquad
\le \|u\|_{L^\infty(B_r)} \int_{\R^n}
\frac{|z|^{2s}\;}{(|x-y|^2+z^2)^{\frac{n+2s}{2}}}\,dy
= C\,\|u\|_{L^\infty(B_r)} ,\end{eqnarray*}
for some~$C>0$.
The latter estimate and~\eqref{89yjkiii098}
imply the desired result, up to renaming the constants.
\end{proof}

\subsection{Functional spaces}\label{F:S}

Given~$r>0$, 
we consider the seminorm
$$ [v]_{\H^s(\B_r)} :=\sqrt{ \int_{\B_r} |z|^a |\nabla v|^2\,dX },$$
with~$a:=1-2s\in(-1,1)$.
We denote by~$\H^s(\B_r)$
the closure of~$C^{\infty}(\B_r)$ with respect to the 
norm
$$ \| v\|_{\H^s(\B_r)} := [ v ]_{\H^s(\B_r)} +
\sqrt{\int_{\B_r} |z|^a |v|^2\,dX}.$$
We also set~$ \H^s_0(\B_r)$
to be the closure of~$C^{\infty}_0(\B_r)$
with respect to the norm above.

For completeness, we point out that the seminorm~$[ \cdot]_{\H^s(\B_r)}$
is indeed a norm on~$\H^s_0(\B_r)$:

\begin{lemma}\label{CHIA67877:L}
If~$v\in \H^s_0(\B_r)$ and~$[v]_{\H^s(\B_r)}=0$ then~$v=0$ a.e. in~$\B_r$.
\end{lemma}

\begin{proof} Let~$v_k\in C^{\infty}_0(\B_r)$ be such that~$\|v_k-v\|_{\H^s(\B_r)}\to0$
as~$k\to+\infty$. Up to subsequences, we may suppose that
\begin{equation}\label{FCW456sG}
{\mbox{$v_k\to v$ a.e. in $\B_r$.}}
\end{equation}
Also, by Proposition 2.1.1 in~\cite{maria},
$$ \left( \int_{\B_r} |z|^a |v_k|^{2\gamma}\,dX\right)^{\frac{1}{2\gamma}}\le
\hat S\, [v_k]_{\H^s(\B_r)},$$
for some~$\gamma>1$ and~$\hat S>0$. Therefore
\begin{eqnarray*}
&& \left( \int_{\B_r} |z|^a |v_k|^{2\gamma}\,dX\right)^{\frac{1}{2\gamma}}
\le \hat S\, \Big( [v_k-v]_{\H^s(\B_r)}+[v]_{\H^s(\B_r)}\Big)\\
&&\qquad
=\hat S\, [v_k-v]_{\H^s(\B_r)}\to 0\end{eqnarray*}
as $k\to+\infty$. This implies that~$v_k\to0$ a.e. in~$\B_r$,
up to subsequences, and therefore~$v=0$ a.e. in~$\B_r$,
thanks to~\eqref{FCW456sG}.
\end{proof}

Given~$\varphi\in \H^s(\B_2)$, we define
$$ \D^\varphi := \big\{
v\in \H^s(\B_1) {\mbox{ s.t. }} v-\varphi\in \H^s_0(\B_1)
\big\}.$$
Now we observe that functions in~$\D^\varphi$ possess a trace
along~$\{z=0\}$. The expert reader may skip this part and go directly to formula \eqref{DE-ha}. 
To give an elementary proof
of this fact (which is rather well known in general,
see e.g. Lemma~3.1
of~\cite{trace} or the references therein), we make
this preliminary observation:

\begin{lemma}\label{TRA}
For any~$v\in C_0^{\infty}(\B_r)$
and any~$x\in B_{\frac{9r}{10}}$,
we define~$T_v (x):=v(x,0)$.
Then, there exists~$C>0$ such that
$$ \| T_v\|_{L^2(B_{{9r}/{10}})}\le C\| v\|_{\H^s(\B_r)}.$$
\end{lemma}

\begin{proof} For any~$x\in B_{\frac{9r}{10}}$,
\begin{eqnarray*}
&& |T_v(x)| = |v(x,0)-v(x,r)|
\\&&\qquad\le \int_0^{r} |\partial_z v(x,z)|\,dz
\le \int_0^{r} |z|^{-\frac{a}{2}} |z|^{\frac{a}{2}} |\nabla v(x,z)|\,dz.\end{eqnarray*}
So, by Cauchy-Schwarz inequality,
for any~$x\in B_{\frac{9r}{10}}$,
$$ |T_v(x)|^2 \le 
\int_0^{r} |z|^{-a} \,dz\,
\int_0^{r} |z|^{a} |\nabla v(x,z)|^2\,dz = C
\int_0^{r} |z|^{a} |\nabla v(x,z)|^2\,dz.$$
Hence we integrate over~$x\in B_{\frac{9r}{10}}$
and the desired result easily follows.
\end{proof}

Now, for any~$w\in \H^s_0(\B_r)$,
we know from the definition of~$\H^s_0(\B_r)$
that there exists a sequence of functions~$w_k\in C_0^{\infty}(\B_r)$
such that~$\| w-w_k\|_{\H^s(\B_r)}\to0$ as~$k\to+\infty$.
By Lemma~\ref{TRA}, we have that
$$ T_{w_k-w_h}(x)=w_k(x,0)-w_h(x,0)
=T_{w_k}(x)-T_{w_h}(x)$$
and so
$$ \| T_{w_k}-T_{w_h}\|_{L^2(B_{{9r}/{10}})}\le C\| w_k-w_h\|_{\H^s(\B_r)}.$$
This means that the sequence $T_{w_k}$ is Cauchy in~$L^2(B_{{9r}/{10}})$,
hence it converges to some
function, denoted as~$T_w$, in~$L^2(B_{{9r}/{10}})$,
which we call the trace of~$w$ along~$\{z=0\}$.
Of course, the trace~$T_w$ is defined up to sets of zero $n$-dimensional Lebesgue
measure, and a different approximating sequence
does produce the same trace:
to check this, take an approximating sequence~$\tilde w_k$
and use again Lemma~\ref{TRA}
to see that
\begin{eqnarray*}&&\| T_{w_k}-T_{\tilde w_k}\|_{L^2(B_{{9r}/{10}})}
\le C\| w_k-\tilde w_k\|_{\H^s(\B_r)}\\ &&\qquad\le
C\| w_k-w\|_{\H^s(\B_r)}+C\| \tilde w_k-w\|_{\H^s(\B_r)},\end{eqnarray*}
hence~$T_{w_k}$ and~$T_{\tilde w_k}$ have the same limit in~$
L^2(B_{{9r}/{10}})$.

Our next goal is to show that we can trace also~$\varphi\in \H^s(\B_2)$
along~$B_{9/10}$. This is not completely obvious since~$\varphi
\not\in\H^s_0(\B_2)$, so the above construction does not apply.
For this, we observe that:

\begin{lemma}\label{PSI12-1}
If~$\varphi_1$, $\varphi_2\in \H^s(\B_2)$ and~$\varphi_1=\varphi_2$
a.e. in~$\B_{5/4}$, then~$\D^{\varphi_1}=\D^{\varphi_2}$.
\end{lemma}

\begin{proof} Let~$v\in \D^{\varphi_1}$. Then~$v-\varphi_1\in\H^s_0(\B_1)$.
Hence there exists a sequence~$w_k\in C^{\infty}_0(\B_1)$ such
that~$\|v-\varphi_1-w_k\|_{\H^s(\B_1)}\to0$
as~$k\to+\infty$. Since~$\varphi_1=\varphi_2$
a.e. in~$\B_{5/4}$, we have that~$\|v-\varphi_1-w_k\|_{\H^s(\B_1)}
=\|v-\varphi_2-w_k\|_{\H^s(\B_1)}$.
As a consequence,~$\|v-\varphi_2-w_k\|_{\H^s(\B_1)}
\to0$
as~$k\to+\infty$, which shows that~$v\in \D^{\varphi_2}$.

The reverse inclusion is completely analogous.
\end{proof}

Now, given~$\varphi\in \H^s(\B_2)$, we can take~$\tau\in C^\infty_0(\B_{3/2})$
with~$\tau=1$ in~$\B_{5/4}$ and consider~$\varphi_o := \tau \varphi$.
By the trace construction in~$\H^s_0(\B_2)$,
we can define the trace~$T_{\varphi_o}$ as a function in~$L^2( B_{2\cdot 9/10})$.
So we define the trace of~$\varphi$ in~$B_{9/10}$ as~$T_\varphi:=T_{\varphi_o}$.
By construction, $T_\varphi\in L^2( B_{9/10})$.
Next observation shows that this definition is independent on the
particular cut-off chosen:

\begin{lemma}
If~$\varphi_1$, $\varphi_2\in \H^s_0(\B_2)$,
with~$\varphi_1=\varphi_2$ a.e. in~$\B_{5/4}$,
then~$T_{\varphi_1}=T_{\varphi_2}$
a.e. in~$B_{9/10}$.
\end{lemma}

\begin{proof} By construction, for any~$i\in\{1,2\}$,
there are sequences~$\varphi_{i,k}\in C^{\infty}_0(\B_2)$
such that~$\|\varphi_i-\varphi_{i,k}\|_{\H^s(\B_2)}
\to0$ as~$k\to+\infty$. 
Let~$\Theta\in C^\infty_0(\B_{5/4})$ with~$\Theta=1$ in~$\B_{11/10}$.
Let also
$$\tilde\varphi_{1,k}:=\varphi_{1,k}+\Theta(\varphi_{2,k}
-\varphi_{1,k}).$$
We claim that
\begin{equation}\label{0iuhgTY6}
\lim_{k\to+\infty}\|\varphi_1-\tilde\varphi_{1,k}\|_{\H^s(\B_2)}=0.
\end{equation}
To prove this, we observe that 
\begin{eqnarray*}
&& |\varphi_1-\tilde\varphi_{1,k}|^2
=\big|\varphi_1-\varphi_{1,k}-\Theta(\varphi_{2,k}-\varphi_{1,k})\big|^2\\
&&\qquad\le C\Big(
|\varphi_1-\varphi_{1,k}|^2 +
|\Theta|^2 |\varphi_{2,k}-\varphi_{1,k}|^2\Big)
\\ &&\qquad\le
C\Big(
|\varphi_1-\varphi_{1,k}|^2 +
+\chi_{\B_{5/4}}|\varphi_{2,k}-\varphi_{1,k}|^2
\Big),\end{eqnarray*}
up to renaming~$C>0$. Hence, since~$\varphi_1=\varphi_2$ a.e. in~$\B_{5/4}$,
$$  \chi_{\B_{5/4}}|\varphi_{2,k}-\varphi_{1,k}|^2
=\chi_{\B_{5/4}}| \varphi_{2,k}-\varphi_2+\varphi_1-\varphi_{1,k} |^2.$$
Therefore
$$
|\varphi_1-\tilde\varphi_{1,k}|^2\le 
C\Big( |\varphi_1-\varphi_{1,k}|^2 +|\varphi_{2,k}-\varphi_2|^2 \Big).  $$
As a consequence, 
\begin{equation}\label{parts}
\lim_{k\to+\infty} \int_{\B_2} |z|^a |\varphi_i-\tilde\varphi_{i,k}|^2\,dX=0.\end{equation}
Moreover, we observe that
\begin{eqnarray*}
&& |\nabla(\varphi_1-\tilde\varphi_{1,k})|^2
=\big|
\nabla(\varphi_1-\varphi_{1,k})-
\nabla(\Theta(\varphi_{2,k}-\varphi_{1,k}))\big|^2\\
&&\qquad\le C\Big(
|\nabla(\varphi_1-\varphi_{1,k})|^2 +
|\nabla\Theta|^2 |\varphi_{2,k}-\varphi_{1,k}|^2
+|\Theta|^2 |\nabla(\varphi_{2,k}-\varphi_{1,k})|^2
\Big)\\ &&\qquad\le
C\Big(
|\nabla(\varphi_1-\varphi_{1,k})|^2 +
\chi_{\B_{5/4}\setminus\B_{11/10}}|\varphi_{2,k}-\varphi_{1,k}|^2
+\chi_{\B_{5/4}}|\nabla(\varphi_{2,k}-\varphi_{1,k})|^2
\Big),\end{eqnarray*}
up to renaming~$C>0$. Hence, since~$\varphi_1=\varphi_2$ a.e. in~$\B_{5/4}$,
$$ \chi_{\B_{5/4}\setminus\B_{11/10}}|\varphi_{2,k}-\varphi_{1,k}|^2
=\chi_{\B_{5/4}\setminus\B_{11/10}}|\varphi_{2,k}-\varphi_2+\varphi_1-\varphi_{1,k}|^2$$
and 
$$ \chi_{\B_{5/4}}|\nabla (\varphi_{2,k}-\varphi_{1,k})|^2
=\chi_{\B_{5/4}}|\nabla ( \varphi_{2,k}-\varphi_2+\varphi_1-\varphi_{1,k} )|^2.$$
Therefore
\begin{eqnarray*}
&& |\nabla(\varphi_1-\tilde\varphi_{1,k})|^2
\\&\le&
C\Big(
|\nabla(\varphi_1-\varphi_{1,k})|^2 +
|\varphi_{2,k}-\varphi_2|^2
+|\varphi_1-\varphi_{1,k}|^2\\ &&\qquad 
+|\nabla(\varphi_{2,k}-\varphi_2)|^2 +|\nabla(\varphi_1-\varphi_{1,k})|^2
\Big),
\end{eqnarray*}
which, after an integration, implies that 
$$ \lim_{k\to+\infty} \int_{\B_2} |z|^a |\nabla(\varphi_i-\tilde\varphi_{i,k})|^2\,dX=0.$$
This and \eqref{parts} give~\eqref{0iuhgTY6}. 

With this, and setting~$\tilde\varphi_{2,k}:=\varphi_{2,k}$,
we have that~$\tilde\varphi_{i,k}\in C^{\infty}_0(\B_2)$,
$\|\varphi_i-\tilde\varphi_{i,k}\|_{\H^s(\B_2)}
\to0$ as~$k\to+\infty$, and, additionally, if~$X\in \B_{11/10}$ then
$$ \tilde\varphi_{1,k}(X)=\tilde\varphi_{2,k}(X).$$
Since~$T_{\varphi_i}$ is the limit in~$L^2(B_{2\cdot(9/10)})$
(and so a.e. in~$B_{2\cdot(9/10)}$, up to subsequences)
of~$T_{\tilde\varphi_{i,k}}$ as~$k\to+\infty$, we have,
for a.e.~$x\in B_{9/10}\subseteq \B_{11/10}\cap \{z=0\}$,
\begin{eqnarray*}
&& T_{\varphi_1}(x)=\lim_{k\to+\infty} T_{\tilde\varphi_{1,k}}(x)=
\lim_{k\to+\infty} {\tilde\varphi_{1,k}}(x,0)
\\ &&\qquad=
\lim_{k\to+\infty} {\tilde\varphi_{2,k}}(x,0)=
\lim_{k\to+\infty} T_{\tilde\varphi_{2,k}}(x)= T_{\varphi_2}(x),
\end{eqnarray*}
as desired.
\end{proof}

Having defined~$T_w$ for any~$w\in \H^s_0(\B_1)$
and~$T_\varphi$ for any~$\varphi\in \H^s(\B_2)$,
we now define the trace of any function~$v\in \D^\varphi$, by setting
$$ T_v := T_{v-\varphi} + T_\varphi.$$
To simplify the notation, given a set~$K\subseteq \B_1\cap\{z=0\}$,
we say that~$v=0$ a.e. in~$K$ to mean that~$T_v=0$
a.e. in~$K$ (i.e.~$v(x,0)=0$ for a.e.~$x\in K$,
in the sense of traces).
We set
\begin{equation}\label{DE-ha}
\D^\varphi_K := \big\{
v\in \D^\varphi
{\mbox{ s.t. }} v=0 {\mbox{ a.e. in }}K
\big\}.
\end{equation}
In some intermediate results, we also need
a slightly more general definition
in which the values attained at~$K$ are not necessarily zero.
For this, given~$\gamma:K\to\R$, we also define
\begin{equation}\label{78hHjjKKKjkK}
\D^\varphi_{K,\gamma} := \big\{
v\in \D^\varphi
{\mbox{ s.t. }} 
v=\gamma {\mbox{ a.e. in }}K
\big\}.\end{equation}
Notice that~$\D^\varphi_{K,\gamma}$ reduces to~$\D^\varphi_K$
when~$\gamma\equiv0$.
The functional structure
of~$\D^\varphi_{K,\gamma}$ that is needed for our
purposes is given by the following result:

\begin{lemma}\label{MARS}
Let~$w_j\in \D^\varphi_{K,\gamma}$ be such that
$$ \sup_{j\in\N} \int_{\B_1} |z|^a |\nabla w_j|^2\,dX<+\infty.$$
Then there exists~$w\in \D^\varphi_{K,\gamma}$ such that, up
to a subsequence,
\begin{equation}\label{PLOK-1}
\lim_{j\to+\infty} \int_{\B_1}|z|^a |w-w_j|^2\,dX =0
\end{equation}
and, for any~$\phi\in \D^\varphi_{K,\gamma}$,
\begin{equation}\label{PLOK-2}
\lim_{j\to+\infty} \int_{\B_1}|z|^a \nabla w_j\cdot\nabla \phi\,dX
=\int_{\B_1}|z|^a \nabla w\cdot\nabla \phi\,dX
.\end{equation}
\end{lemma}

\begin{proof} First, we use
Lemma~2.1.2 in~\cite{maria} and we obtain that
there exists~$w$ (with finite weighted Lebesgue norm)
such that~\eqref{PLOK-1} holds true.
Then, by Theorem~1.31 in~\cite{martio},
we obtain~\eqref{PLOK-2}.
It remains to show that
\begin{equation}\label{PLOK-3}
w\in \D^\varphi_{K,\gamma}.
\end{equation}
To this goal, we first observe that~$\H^s_0(\B_1)$ is closed
(with respect to~$\|\cdot\|_{\H^s(\B_1)}$) and convex.
Hence~$\D^\varphi$ is also closed and convex,
and then so is~$\D^\varphi_{K,\gamma}$. Therefore~\eqref{PLOK-3}
follows from~\eqref{PLOK-1}, \eqref{PLOK-2}
and Theorem~1.30 in~\cite{martio} (applied here with~${\mathcal{K}}:=
\D^\varphi_{K,\gamma}$).
\end{proof}

Now we define
$$ \E(v):=\int_{\B_1}|z|^a |\nabla v|^2\,dX.$$
Then we have:

\begin{theorem}\label{MIN-0}
Assume that
\begin{equation}\label{NON-0}
\D^\varphi_{K,\gamma}\ne\varnothing.
\end{equation}
Then there exists a unique~$\Phi^\varphi_{K,\gamma}\in \D^\varphi_{K,\gamma}$
such that
$$ \E(\Phi^\varphi_{K,\gamma}) = \min_{v\in \D^\varphi_{K,\gamma}} \E(v).$$
In particular, taking~$\gamma\equiv0$, we have that if~$\D^\varphi_K\ne\varnothing$
then there exists a unique~$\Phi^\varphi_K\in \D^\varphi_K$ such that
$$ \E(\Phi^\varphi_K) = \min_{v\in \D^\varphi_K} \E(v).$$
\end{theorem}

\begin{proof} Let
$$ \iota:=\inf_{v\in \D^\varphi_{K,\gamma}} \E(v).$$
We take a minimizing sequence~$w_j\in\D^\varphi_{K,\gamma}$
such that
\begin{equation}\label{7udvf67890UU}\E(w_j)\le\iota + e^{-j}.\end{equation}
By Lemma~\ref{MARS}, up to a subsequence
we have that there exists~$w\in \D^\varphi_{K,\gamma}$ such that
$$ \lim_{j\to+\infty} \int_{\B_1}|z|^a \nabla w_j \cdot\nabla\phi\,dX=
\int_{\B_1}|z|^a \nabla w\cdot\nabla\phi\,dX,$$
for every~$\phi\in \D^\varphi_{K,\gamma}$.
In particular,
\begin{eqnarray*}
0 &\le& \liminf_{j\to+\infty} \int_{\B_1}|z|^a |\nabla (w_j-w)|^2\,dX
\\ &=&
\liminf_{j\to+\infty} \int_{\B_1}|z|^a |\nabla w_j|^2\,dX
+\int_{\B_1}|z|^a |\nabla w|^2\,dX
-2\int_{\B_1}|z|^a \nabla w_j\cdot\nabla w\,dX\\
&=&\liminf_{j\to+\infty} \int_{\B_1}|z|^a |\nabla w_j|^2\,dX
-\int_{\B_1}|z|^a |\nabla w|^2\,dX\\
&=&\liminf_{j\to+\infty}\E(w_j)-\E(w).
\end{eqnarray*}
By inserting this into~\eqref{7udvf67890UU}
we obtain that
$$ \E(w)\le \liminf_{j\to+\infty} \E(w_j)
\le
\liminf_{j\to+\infty}
\iota+e^{-j}
=\iota.$$
This shows that~$w$ is the desired minimizer.

Now we show that the minimizer is unique.
The proof relies on a standard convexity argument,
we give the details for the facility of the reader.
Suppose that we have two minimizers~$w_1$ and~$w_2$, and let~$w:=(w_1+w_2)/2$.
Notice that~$w\in\D^\varphi_{K,\gamma}$ by the
convexity of the space, hence
$$ \E(w_1)=\E(w_2)\le \E(w).$$
Also~$w_1-w_2\in\H^s_0(\B_1)$, thus
\begin{eqnarray*}
[ w_1-w_2]_{\H^s(\B_1)}^2 &=& \int_{\B_1} |z|^a |\nabla(w_1-w_2)|^2\,dX
\\ &=& \int_{\B_1} |z|^a \big( |\nabla w_1|^2 +|\nabla w_2|^2
-2\nabla w_1\cdot\nabla w_2\big)
\,dX \\
&=& 
\int_{\B_1} |z|^a \big( 2|\nabla w_1|^2 +2|\nabla w_2|^2
-|\nabla (w_1+w_2)|^2\big)\,dX
\\ &=& 2\E(w_1) +2\E(w_2) -4\E(w)
\\ &\le&0.
\end{eqnarray*}
This, together with Lemma~\ref{CHIA67877:L},
shows that~$w_1=w_2$ and so it completes the proof of the uniqueness
claim.
\end{proof}

{F}rom now on, we will implicitly assume that~$\D^\varphi_K\ne\varnothing$.
Then, the minimizer~$\Phi^\varphi_K$ introduced
in Theorem~\ref{MIN-0} is the fractional harmonic replacement
that we consider in this paper. Roughly speaking, it is a minimizer
with boundary datum~$\varphi$ of a fractional energy in the extended
variables under the additional condition of vanishing in the set~$K$.

\subsection{Basic properties of the fractional harmonic replacement}

In this subsection,
we prove some simple, but useful, properties of the
fractional harmonic replacement, such as symmetry and
harmonicity properties and maximum principles.

We remark that the fractional harmonic replacement
is defined in a whole~$(n+1)$-dimensional set.
This can be translated into subset of the halfspace~$\R^{n+1}_+$
if the boundary datum
is even in~$z$, as the forthcoming Lemma~\ref{EVE}
will point out. 

\begin{lemma}\label{EVE}
If~$\varphi(x,-z)=\varphi(x,z)$ then~$\Phi^\varphi_{K,\gamma}(x,-z)=
\Phi^\varphi_{K,\gamma}(x,z)$.
\end{lemma}

\begin{proof} We let~$\Psi(x,z):=\Phi^\varphi_{K,\gamma}(x,-z)$.
Then~$\Psi\in\D^\varphi_{K,\gamma}$. Furthermore
$$ \int_{\B_1}|z|^a |\nabla\Psi|^2\,dX=
\int_{\B_1}|z|^a |\nabla\Phi^\varphi_{K,\gamma}(x,-z)|^2\,dX
=\int_{\B_1}|z|^a |\nabla\Phi^\varphi_{K,\gamma}(x,z)|^2\,dX,$$
hence~$\Psi$ is also a minimizer for~$\E$ in~$\D^\varphi_{K,\gamma}$.
By the uniqueness result in Theorem~\ref{MIN-0},
we conclude that~$\Psi=\Phi^\varphi_{K,\gamma}$.
\end{proof}

Now we write~$\D^0_K$ to mean the functional space~$\D^\varphi_K$
when~$\varphi\equiv0$. In this notation, we have
that the fractional harmonic replacement is orthogonal to~$\D^0_K$,
as stated in the following result:

\begin{lemma}
For every~$\psi \in\D^0_K$,
\begin{equation}\label{0PouyghjYU-1}
\int_{\B_1}|z|^a \nabla\Phi^\varphi_{K,\gamma}
\cdot\nabla\psi\,dX=0\end{equation}
and
\begin{equation}\label{0PouyghjYU-2}
\E(\Phi^\varphi_{K,\gamma}\pm \psi)=\E(\Phi^\varphi_{K,\gamma})+\E(\psi).
\end{equation}
\end{lemma}

\begin{proof} Notice that for every~$\eps\in(-1,1)$, we have that~$
\Phi^\varphi_{K,\gamma}+\eps\psi\in \D^\varphi_{K,\gamma}$, therefore~$
\E(\Phi^\varphi_{K,\gamma}+\eps\psi)-
\E(\Phi^\varphi_{K,\gamma})\ge0$ and then~\eqref{0PouyghjYU-1} follows.

Then, using~\eqref{0PouyghjYU-1},
\begin{eqnarray*}
&& \E(\Phi^\varphi_{K,\gamma}\pm \psi)-\E(\Phi^\varphi_{K,\gamma})-\E(\psi) \\
&=& \int_{\B_1}|z|^a \big[
|\nabla\Phi^\varphi_{K,\gamma}|^2+|\nabla\psi|^2
\pm 2\nabla\Phi^\varphi_{K,\gamma}
\cdot\nabla\psi
\big]\,dX\\ &&\qquad-
\int_{\B_1}|z|^a 
|\nabla\Phi^\varphi_{K,\gamma}|^2\,dX-\int_{\B_1}|z|^a
|\nabla\psi|^2\,dX
\\ &=&0,
\end{eqnarray*}
that establishes~\eqref{0PouyghjYU-2}.
\end{proof}

Now we show that the fractional harmonic extension is indeed
``harmonic'' outside the constrain, i.e. it satisfies
a weighted elliptic equation in the interior of~$\B_1\setminus K$.
The precise statement goes as follows:

\begin{lemma}\label{HARM:0p}
We have that
\begin{equation}\label{9YHBVf}
{\rm div}\, (|z|^a \nabla\Phi^\varphi_{K,\gamma})=0\end{equation}
in the interior of~$\B_1\setminus K$, in the distributional sense.
\end{lemma}

\begin{proof} Let~${\mathcal{N}}$ be an open set contained in~$
\B_1\setminus K$. Let~$\psi\in C^\infty_0( {\mathcal{N}})$.
Then~$\psi=0$ in~$K$ and so~$\psi\in\D^0_K$. Accordingly,
by~\eqref{0PouyghjYU-1},
$$ \int_{\B_1}|z|^a \nabla\Phi^\varphi_{K,\gamma}\cdot\nabla\psi\,dX=0,$$
which establishes \eqref{9YHBVf} in the distributional sense.
\end{proof}

The forthcoming two results in Lemmata~\ref{MPL-1}
and~\ref{MPL-2}
provide uniform bounds
on~$\Phi^\varphi_K$ by Maximum Principle. To this goal,
we need the ancillary observations in the following Lemmata~\ref{AUX:0ugBIS}--\ref{AUX:0ug:NEG}: 

\begin{lemma}\label{AUX:0ugBIS}
Let~$c\in\R$ and~$\phi\in\H^s(\B_1)$.
Let~$\phi_k\in\H^s(\B_1)$ be a sequence such that
\begin{equation}\label{CV-XC-0-PRE}
\lim_{k\to+\infty} \| \phi-\phi_k\|_{\H^s(\B_1)}=0.
\end{equation}
Let~$\psi:=(\phi-c)^+$ and~$\psi_k:=(\phi_k-c)^+$. Then, up to a subsequence,
$$ \lim_{k\to+\infty} \| \psi-\psi_k\|_{\H^s(\B_1)}=0.$$
\end{lemma}

\begin{proof} First, we observe that, up to a subsequence,
$\phi_k\to\phi$ a.e. in~$\B_1$. Accordingly
\begin{equation}\label{0P}
\limsup_{k\to+\infty} 
\chi_{ \{\phi_k>c\ge\phi\} }\le \chi_{ \{ \phi=c\} }
\;{\mbox{ and }}\;
\limsup_{k\to+\infty} 
\chi_{ \{\phi>c\ge\phi_k\} }\le\chi_{ \{ \phi=c\} }
\end{equation}
a.e. in~$\B_1$.
Also, for any domain~${\mathcal{N}}$ compactly contained
in~$\B_1\setminus \{z=0\}$, we have that~$\phi\in W^{1,1}_{\rm loc}
({\mathcal{N}})$ and so, by Stampacchia's Theorem (see e.g. Theorem~6.19
in~\cite{lieb-loss}), it follows that~$\nabla\phi=0$ a.e. in~$\{
\phi=c\}$, and so
\begin{equation*}
{\mbox{$|z|^a|\nabla\phi|^2 \chi_{ \{
\phi=c\} }=0$ a.e. in~$\B_1$.}}\end{equation*}
Therefore, by~\eqref{0P},
\begin{eqnarray*}
&& \lim_{k\to+\infty} |z|^a|\nabla\phi|^2 
\chi_{ \{\phi_k>c\ge\phi\} }=0\\
{\mbox{and }}
&& \lim_{k\to+\infty} |z|^a|\nabla\phi|^2 
\chi_{ \{\phi>c\ge\phi_k\} }=0.
\end{eqnarray*}
Consequently, by the Dominated Convergence Theorem,
\begin{equation}\label{HjnmUI00}
\begin{split}
&\lim_{k\to+\infty} 
\int_{\B_1\cap \{ \phi>c\ge\phi_k\}
}|z|^a |\nabla\phi |^2\,dX=0\\
{\mbox{and }}\;&\lim_{k\to+\infty}
\int_{\B_1\cap \{ \phi_k>c\ge\phi\}
}|z|^a |\nabla\phi |^2\,dX
=0.\end{split}
\end{equation}
Moreover, by Corollary~2.1 in~\cite{FKS},
\begin{eqnarray*}
&& [\psi-\psi_k]_{\H^s(\B_1)}^2
=\int_{\B_1}|z|^a |\nabla\psi -\nabla\psi_k|^2\,dX\\
\\&&\qquad=
\int_{\B_1}|z|^a |\nabla(\phi -c)^+ -\nabla(\phi_k -c)^+
|^2\,dX \\
\\&&\qquad=
\int_{\B_1\cap\{ \phi>c\}\cap
\{ \phi_k>c\}}
|z|^a |\nabla\phi -\nabla\phi_k |^2\,dX
\\&&\qquad\qquad+
\int_{\B_1\cap \{ \phi>c\ge\phi_k\}
}|z|^a |\nabla\phi |^2\,dX
\\&&\qquad\qquad+
\int_{\B_1\cap \{ \phi_k>c\ge\phi\}}
|z|^a |\nabla\phi_k|^2\,dX.
\end{eqnarray*}
We also observe that
$$ |\nabla\phi_k|^2 \le 2\Big(|\nabla\phi_k-\nabla\phi |^2
+|\nabla\phi |^2\Big)$$
and therefore
\begin{eqnarray*}&& [\psi-\psi_k]_{\H^s(\B_1)}^2
\le 3
\int_{\B_1} |z|^a |\nabla\phi -\nabla\phi_k |^2\,dX
\\ &&\quad+\int_{\B_1\cap \{ \phi>c\ge\phi_k\}
}|z|^a |\nabla\phi |^2\,dX
+2\int_{\B_1\cap \{ \phi_k>c\ge\phi\}
}|z|^a |\nabla\phi |^2\,dX.\end{eqnarray*}
{F}rom this, \eqref{CV-XC-0-PRE}
and~\eqref{HjnmUI00}, we get
\begin{equation}\label{K:l67:001:L}
\lim_{k\to+\infty}[\psi-\psi_k]_{\H^s(\B_1)}^2
\le0.\end{equation}
Now we observe that~$|z|^a |\phi-\phi_k|^2\to0$
in~$L^1(\B_1)$, thanks to~\eqref{CV-XC-0-PRE}.
Therefore (see e.g. Theorem~4.9(b) in~\cite{brezis}), we know that,
up to a subsequence,
$$ |z|^a |\phi-\phi_k|^2 \le h,$$
for every~$k\in\N$, with~$h\in L^1(\B_1)$. As a consequence,
$$ |z|^{\frac{a}{2}} |\phi_k|\le
|z|^{\frac{a}{2}} |\phi-\phi_k|
+|z|^{\frac{a}{2}} |\phi|\le \sqrt{h}+|z|^{\frac{a}{2}} |\phi|.$$
Consequently,
$$ |z|^{\frac{a}{2}} |\psi-\psi_k|
\le |z|^{\frac{a}{2}} \Big( |\phi|+|\phi_k|+2|c|\Big)
\le 2|z|^{\frac{a}{2}} \Big( |\phi|+|c|\Big) 
+\sqrt{h}$$
and thus
$$ |z|^{a} |\psi-\psi_k|^2
\le C\,\Big[|z|^a\Big( |\phi|^2+c^2\Big)+h\Big]=:g,$$
with~$g\in L^1(\B_1)$. So, by
the Dominated Convergence Theorem,
$$ \lim_{k\to+\infty}\int_{\B_1}|z|^a|\psi-\psi_k|^2\,dX =0. $$
This formula and~\eqref{K:l67:001:L}
imply the desired result.
\end{proof}

We need now a technical modification
of Lemma~\ref{AUX:0ugBIS}. Namely, given~$\phi\in\H^s(\B_1)$,
in order to approximate~$\phi^+$
in~$\H^s(\B_1)$ it is not always convenient to consider
the positive parts of the approximating sequence
(as done in Lemma~\ref{AUX:0ugBIS}), since taking positive
parts may decrease the regularity of the smooth functions.
To avoid this, we introduce a smooth modification
of an approximating sequence, which still converges to
the positive part in the limit.
The key step in this procedure is given by the following result:

\begin{lemma}\label{5599}
Let~$\phi\in\H^s(\B_1)$ and fix~$\eps>0$.
Then, there exist~$\overline\theta_\eps$,
$\underline\theta_\eps\in C^\infty(\R)$ such that~$\underline\theta_\eps(t)\le t^+\le \overline\theta_\eps(t)$
for any~$t\in\R$ and
\begin{equation}\label{0iF-G} \|\phi^+ -\overline\theta_\eps(\phi)\|_{\H^s(\B_1)}
+\|\phi^+ -\underline\theta_\eps(\phi)\|_{\H^s(\B_1)}\le\eps.\end{equation}
\end{lemma}

\begin{proof} Let~$\tau\in C^\infty(\R,[0,1])$ such that~$\tau(t)=0$
for any~$t\le 1/2$, and~$\tau(t)=1$ for any~$t\ge 3/4$.
Let also~$\Theta(t):=t\,\tau(t)$ and
$$ \underline\theta_\eps(t):= \eps \Theta\left( \frac{t}{\eps}\right).$$
By construction, $\Theta(t)\le t^+$ and so~$\underline\theta_\eps(t)\le t^+$
for any~$t\in\R$. 

Moreover,
\begin{equation}\label{UI-01-00}
|\underline\theta_\eps'|\le C,\end{equation}
for some~$C>0$, and
\begin{equation}\label{UI-01-01}
{\mbox{$ \underline\theta_\eps(t) = t^+$ for any~$|t|\ge\eps$.}}\end{equation}

Now we take a nondecreasing function~$\mu\in C^\infty(\R)$
such that~$\mu(t)=0$ if~$t\le -1/100$, $\mu(t)\in(0,1)$ for any~$t\in(-1/100,1/100)$
and~$\mu(t)=1$ for any~$t\ge1/100$. 
Notice that 
\begin{equation}\label{intmu}
\iota:=\int_{-\infty}^{\frac{1}{100}}\mu(t)\,dt \le\frac{1}{50}.
\end{equation}
For any~$r>0$, we define
$$ \mu_r(t):=\mu\left(t-\frac{99}{100}+r\right).$$
We observe that~$\mu_r(t)=0$ if~$t\le (98/100)-r$, $\mu_r(t)\in(0,1)$ 
for any~$t\in((98/100)-r,1-r)$
and~$\mu_r(t)=1$ for any~$t\ge 1-r$. 

We claim that 
\begin{equation}\label{uffa08997889}
{\mbox{there exists $r\in[0,1]$ such that }} 
\int_{-\infty}^{1} \mu_r(t)\,dt =1.
\end{equation}
To prove this, notice that, using the change of variable~$\tilde t=t-\frac{99}{100}+r$ 
and recalling~\eqref{intmu}, 
\begin{eqnarray*}
&& \int_{-\infty}^{1} \mu_r(t)\,dt = \int_{-\infty}^{1} \mu\left(t-\frac{99}{100}+r\right)\,dt\\
&&\qquad = \int_{-\infty}^{\frac{1}{100}+r} \mu(\tilde t)\,d\tilde t 
= \int_{-\infty}^{\frac{1}{100}} \mu(\tilde t)\,d\tilde t + 
\int_{\frac{1}{100}}^{\frac{1}{100}+r} \mu(\tilde t)\,d\tilde t\\
&&\qquad =\iota + \int_{\frac{1}{100}}^{\frac{1}{100}+r} 1\,d\tilde t 
=\iota + r.
\end{eqnarray*}
Now, if~$r=0$ then~$\iota\le 1/50$, thanks to~\eqref{intmu}, 
and if~$r=1$ then~$\iota +1\ge 1$, since~$\iota\ge0$. 
So, by continuity, we obtain the claim in~\eqref{uffa08997889}. 

Notice that the parameter~$r$ given by~\eqref{uffa08997889} will be considered
as fixed from now on. We define
$$ T(t):= \int_{-\infty}^t \mu_r (\rho)\,d\rho.$$
We claim that
\begin{equation}\label{UI-017899}
{\mbox{$T(t) = t^+$ for any~$|t|\ge 1$.}}\end{equation}
Indeed, if~$t\le -1$ then~$t\le (98/100)-r$ and so we have that~$T(t)=0=t^+$,
since the integrand vanishes. Also, if~$t\ge 1$ then
$$ T(t)=\int_{-\infty}^1 \mu_r (\rho)\,d\rho
+\int_{1}^t \mu_r (\rho)\,d\rho =
1+\int_{1}^t 1\,d\rho=t,$$
where~\eqref{uffa08997889} was used. This proves~\eqref{UI-017899}.

We also claim that
\begin{equation}\label{BISUI-017899}
{\mbox{$T(t) \ge t^+$ for any~$t\in\R$.}}\end{equation}
To prove it, we notice that it is enough to consider the case~$t\in(-1,1)$,
in view of~\eqref{UI-017899}. Moreover, $T(t)\ge 0=t^+$
for any~$t\le0$, so we can focus on the case~$t\in(0,1)$.
For this, for any~$t\in(0,1)$, we let~$H(t):=T(t)-t^+=T(t)-t$.
Then
$$ H'(t)=T'(t) -1 =\mu_r(t)-1\le0.$$
Therefore, for any~$t\in(0,1)$,
$$ T(t)-t^+ = H(t) \ge H(1) = T(1)-1=0,$$
due to~\eqref{UI-017899}, and this completes the proof of~\eqref{BISUI-017899}.

Now we define
$$ \overline\theta_\eps(t):= \eps T \left( \frac{t}{\eps}\right).$$
{F}rom~\eqref{BISUI-017899}, we know that~$\overline\theta_\eps(t)\ge t^+$
for any~$t\in\R$. Also,
\begin{equation}\label{UI-01-00-bis}
|\overline\theta_\eps'|\le C,\end{equation}
for some~$C>0$, and we deduce from~\eqref{UI-017899} that 
\begin{equation}\label{UI-01-01-bis}
{\mbox{$ \overline\theta_\eps(t) = t^+$ for any~$|t|\ge \eps$.}}\end{equation}
Having completed the construction of~$\overline\theta_\eps$
and~$\underline\theta_\eps$, we now prove~\eqref{0iF-G}.
To this goal, by Lemma~2.1 in~\cite{FKS}, we have that~$\nabla( \underline\theta_\eps(\phi))
= \underline\theta_\eps'(\phi)\nabla \phi$, therefore
\begin{equation}\label{STG}\begin{split}
& \|\phi^+ - \underline\theta_\eps(\phi)\|_{\H^s(\B_1)}^2 =
\int_{\B_1} |z|^a \big| \nabla\phi^+ -  \underline\theta_\eps'(\phi)\nabla \phi\big|^2\,dX
\\ &\qquad=
\int_{\B_1\cap \{ |\phi|<\eps \}
} |z|^a \big| \nabla\phi^+ -  \underline\theta_\eps'(\phi)\nabla \phi\big|^2\,dX,
\end{split}\end{equation}
since the other contributions cancel, thanks to~\eqref{UI-01-01}.

We also use~\eqref{UI-01-00}
to see that~$|z|^a \big| \nabla\phi^+ -  \underline\theta_\eps'(\phi)\nabla \phi\big|^2
\,\chi_{ \{ |\phi|<\eps \} }\le
C \, |z|^a |\nabla\phi|^2\in L^1(\B_1)$, since~$\phi\in\H^s(\B_1)$,
therefore, by the Dominated Convergence Theorem
and the Theorem of Stampacchia (see e.g. Theorem~6.19
in~\cite{lieb-loss}), we have
$$ \lim_{\eps\to0}
\int_{\B_1\cap \{ |\phi|<\eps \}
} |z|^a \big| \nabla\phi^+ -  \underline\theta_\eps'(\phi)\nabla \phi\big|^2\,dX
\le C
\int_{\B_1\cap \{ \phi=0 \}
} |z|^a |\nabla\phi|^2 \,dX=0.$$
This and~\eqref{STG} give that
\begin{equation}\label{POLKAST-0}
\lim_{\eps\to0} [\phi^+ -\underline\theta_\eps(\phi)]_{\H^s(\B_1)}^2 =0.
\end{equation}
Now we observe that~$|\underline\theta_\eps(t)|\le C(1+|t|)$,
due to~\eqref{UI-01-00}, and therefore, by 
the Dominated Convergence Theorem,
$$ \lim_{\eps\to0} \int_{\B_1}|z|^a|\phi^+ -\underline\theta_\eps(\phi)|^2\,dX =0.$$ 
This and \eqref{POLKAST-0} imply that
\begin{equation}\label{POLKAST}
\lim_{\eps\to0} \|\phi^+ -\underline\theta_\eps(\phi)\|_{\H^s(\B_1)}^2 =0.
\end{equation}
In a similar way (using~\eqref{UI-01-00-bis} and~\eqref{UI-01-01-bis}
instead of~\eqref{UI-01-00} and~\eqref{UI-01-01}), 
we obtain that
$$ \lim_{\eps\to0} \|\phi^+ -\overline\theta_\eps(\phi)\|_{\H^s(\B_1)}^2 =0.$$
This and~\eqref{POLKAST}
give \eqref{0iF-G} (up to renaming~$\eps$).
\end{proof}

As a consequence of Lemmata~\ref{AUX:0ugBIS}
and~\ref{5599} we have the following smooth
approximation result for the positive part:

\begin{corollary}\label{coro:smooth}
Let~$c\in\R$ and~$\phi\in\H^s(\B_1)$.
Let~$\phi_k\in\H^s(\B_1)$ be a sequence such that
$$ \lim_{k\to+\infty} \| \phi-\phi_k\|_{\H^s(\B_1)}=0.$$
Then, there exist sequences of functions~$\overline\theta_k$, $\underline\theta_k\in C^\infty(\R)$
such that~$\underline\theta_k(t)\le t^+\le\overline\theta_k(t)$ for any~$t\in\R$
and 
\begin{equation}\label{yu78JK}
\lim_{k\to+\infty} \| (\phi-c)^+ -\underline\theta_k(\phi_k-c)\|_{\H^s(\B_1)}=0 \end{equation}
and 
\begin{equation}\label{yu78JK-bis}
\lim_{k\to+\infty} \| (\phi-c)^+ -\overline\theta_k(\phi_k-c)\|_{\H^s(\B_1)}=0. \end{equation}
\end{corollary}

\begin{proof} First we use
Lemma~\ref{AUX:0ugBIS} to say that
$$ \lim_{k\to+\infty} \| (\phi-c)^+ - (\phi_k-c)^+\|_{\H^s(\B_1)}=0.$$
Now, fixed~$k\in\N$, we use Lemma~\ref{5599} to 
find~$\overline\theta_k$, $\underline\theta_k\in C^\infty(\R)$
such that~$\underline\theta_k(t)\le t^+\le\overline\theta_k(t)$ for any~$t\in\R$ and
$$ \| (\phi_k-c)^+ - \overline\theta_k(\phi_k-c)\|_{\H^s(\B_1)} 
+\| (\phi_k-c)^+ - \underline\theta_k(\phi_k-c)\|_{\H^s(\B_1)}\le e^{-k}.$$
These considerations and the triangle inequality imply~\eqref{yu78JK} and~\eqref{yu78JK-bis},
as desired.
\end{proof}

With this, we can now prove the following result: 

\begin{lemma}\label{AUX:0ug}
Let~$g$, $\varphi\in\H^s(\B_2)$ with~$g-\varphi\in\H^s_0(\B_1)$.
Let also~$c\ge\displaystyle\sup_{\B_1}\varphi$. Then~$(g-c)^+\in\H^s_0(\B_1)$.
\end{lemma}

\begin{proof}
By construction~$g-c\in\H^s(\B_1)$.
Thus, by Corollary~2.1 in~\cite{FKS}, we have that~$(g-c)^+\in
\H^s(\B_1)$. 
Moreover, there exist sequences~$f_k\in C^{\infty}_0(\B_1)$
and~$\varphi_k\in C^{\infty}(\B_1)$ such that~$f_k\to g-\varphi$
and~$\varphi_k\to\varphi$ in~$\H^s(\B_1)$ as~$k\to+\infty$,
respectively.

Now, we define~$\tilde\varphi_k:=\varphi_k-\overline\theta_k(\varphi_k-c)$, 
where $\overline\theta_k$ is the smooth function given by Corollary \ref{coro:smooth}. 
Notice that~$\tilde\varphi_k\in C^{\infty}(\B_1)$. Also,
by Corollary~\ref{coro:smooth}, we have that~$\overline\theta_k(\varphi_k-c)\to
(\varphi-c)^+=0$ in~$\H^s(\B_1)$, therefore~$\tilde\varphi_k\to
\varphi$ in~$\H^s(\B_1)$, as $k\to+\infty$.

Now we define~$h_k:= \underline\theta_k( f_k+\tilde\varphi_k-c-e^{-k})$, 
where $\underline\theta_k$ is given by Corollary \ref{coro:smooth}.
Notice that~$h_k\in C^{\infty}(\B_1)$.
Also, the support of~$h_k$ is compactly contained inside~$\B_1$,
since~$\tilde\varphi_k\le \varphi_k-(\varphi_k-c)^+=\min\{\varphi_k,c\}\le c$ 
(recall that $\overline\theta_k(t)\ge t^+$ for any $t\in\R$)
and the support of~$f_k$ is compactly contained inside~$\B_1$.
Therefore, we have that~$h_k\in C^{\infty}_0(\B_1)$.
Also, by Corollary~\ref{coro:smooth}, we have that~$h_k\to
( (g-\varphi)+\varphi-c)^+ = (g-c)^+$
in~$\H^s(\B_1)$.
This implies that~$(g-c)^+\in\H^s_0(\B_1)$.
\end{proof}

For further reference, we point out that a statement analogous to
Lemma~\ref{AUX:0ug} holds when the positive part is replaced
with the negative part of the functions:

\begin{lemma}\label{AUX:0ug:NEG}
Let~$g$, $\varphi\in\H^s(\B_2)$ with~$g-\varphi\in\H^s_0(\B_1)$.
Let also~$c\le\displaystyle\inf_{\B_1}\varphi$. Then~$(g-c)^-\in\H^s_0(\B_1)$.
\end{lemma}

Now we establish pointwise bounds, from above and below,
of the fractional harmonic replacement:

\begin{lemma}\label{MPL-1}
We have that
$$ \Phi^\varphi_{K,\gamma} \le \max\left\{
\sup_{\B_1} \varphi,\;\sup_{K}\gamma\right\}.$$
\end{lemma}

\begin{proof} Let
$$c:=
\max\left\{
\sup_{\B_1} \varphi,\;\sup_{K}\gamma\right\}$$
and~$\psi:=(\Phi^\varphi_{K,\gamma} -c)^+$.
By Lemma~\ref{AUX:0ug}, we know that~$\psi\in\H^s_0(\B_1)$.
Also, a.e. in~$K$,
$$ \psi=(\Phi^\varphi_{K,\gamma} -c)^+=(\gamma-c)^+=0$$
in the sense of traces, hence~$\psi\in
\D^0_K$. As a consequence, using~\eqref{0PouyghjYU-1},
$$ 0=\int_{\B_1}|z|^a \nabla\Phi^\varphi_{K,\gamma}\cdot\nabla\psi\,dX=
\int_{\B_1\cap\{ \Phi^\varphi_{K,\gamma}>c\}}|z|^a
|\nabla\Phi^\varphi_{K,\gamma}|^2\,dX,$$
which gives the desired result.
\end{proof}

\begin{lemma}\label{MPL-2}
If~$\varphi\ge0$ and~$\gamma\ge0$, then~$\Phi^\varphi_{K,\gamma}\ge0$.
\end{lemma}

\begin{proof} Let~$\psi:=(-\Phi^\varphi_{K,\gamma})^+ =
(\Phi^\varphi_{K,\gamma})^-$.
By Corollary~2.1 in~\cite{FKS} we have that~$\psi\in \H^s(\B_1)$,
and, using Lemma~\ref{AUX:0ug:NEG} with~$c:=0$, we have that~$\psi\in
\H^s_0(\B_1)$.
Also, a.e. in~$K$, we have that~$\psi=(\Phi^\varphi_{K,\gamma})^-=
(\gamma)^-=0$ in the trace sense.
As a consequence, $\psi\in\D^0_K$,
thus we can use~\eqref{0PouyghjYU-1} and conclude that
$$ 0=
\int_{\B_1}|z|^a \nabla\Phi^\varphi_{K,\gamma}\cdot\nabla\psi\,dX=
-\int_{\B_1\cap\{ \Phi^\varphi_{K,\gamma}<0\}}|z|^a |\nabla\Phi^\varphi_{K,\gamma}|^2\,dX,$$
which gives the desired result.
\end{proof}

\subsection{Relaxation of the functional spaces
and subharmonicity properties}

The purpose of this subsection is to relax the functional
prescription in the space~$\D^\varphi_K$ by allowing approximating sequences
to take also negative values in~$K$. This observation
will be exploited to deduce subharmonicity
properties of~$\Phi^\varphi_K$ and it will also play a role
in the proof of the monotonicity statement of Theorem~\ref{MONOTONE}.
For this scope,
we define
\begin{equation}\label{9hjkgGGHJ}
\tilde\D^\varphi_K := \big\{
v\in \D^\varphi
{\mbox{ s.t. }}
v\le 0 {\mbox{ a.e. in }}K\big\}.\end{equation}
The reader may compare this definition with~\eqref{DE-ha}:
the only difference is that in~\eqref{DE-ha}
the function is forced to vanish on~$K$,
while in the latter setting it 
can also attain negative values on~$K$.
Of course, $\tilde\D^\varphi_K \supseteq \D^\varphi_K$,
therefore
$$ \inf_{ v\in\tilde\D^\varphi_K }\E(v)\le
\min_{v\in \D^\varphi_K }\E(v)=\E(\Phi^\varphi_K).$$
We will show that in fact equality holds if~$\varphi\ge0$:

\begin{lemma}\label{0iGHBCcvL}
If~$\varphi\ge0$, then
$$ \min_{ v\in\tilde\D^\varphi_K }\E(v)=
\min_{v\in \D^\varphi_K }\E(v)=\E(\Phi^\varphi_K).$$
\end{lemma}

\begin{proof} Let~$v\in \tilde\D^\varphi_K$. Since~$|\nabla v^+|\le
|\nabla v|$, we have that~$\E(v^+)\le\E(v)$.
So, to prove the desired result, we only have to show that
\begin{equation}\label{0iGHBCcv}
v^+\in\D^\varphi_K.
\end{equation}
For this, we note that~$v^+\in \H^s(\B_1)$,
thanks to Corollary~2.1 in~\cite{FKS}.
Now we claim that
\begin{equation}\label{89tTyH}
v^+ -\varphi\in\H^s_0(\B_1).
\end{equation}
For this, we use the sequences~$f_k\in C^{\infty}_0(\B_1)$
and~$\varphi_k\in C^{\infty}(\B_1)$ that converge, respectively,
to~$v-\varphi$ and~$\varphi$ in~$\H^s(\B_1)$, as~$k\to+\infty$.

We define~$g_k:=f_k +\overline\theta_k(\varphi_k)$, where $\overline\theta_k$ 
is given by Corollary~\ref{coro:smooth}.
Hence, by Corollary~\ref{coro:smooth}, we know that~$\overline\theta_k(\varphi_k)\to
\varphi^+=\varphi$ in~$\H^s(\B_1)$.
Therefore~$g_k\to (v-\varphi) +\varphi=v$
in~$\H^s(\B_1)$.

As a consequence, using again Corollary~\ref{coro:smooth}, 
we obtain that~$\overline\theta_k(g_k)\to v^+$ in~$\H^s(\B_1)$.

Let now~$h_k:= \overline\theta_k(g_k) -\overline\theta_k\big( 
\overline\theta_k(\varphi_k)\big)$. We have that~$h_k\to
v^+ -\varphi$. We also notice that~$f_k=0$
outside a compact subset~${\mathcal{K}}_k$ contained inside~$\B_1$.
Hence~$g_k=\overline\theta_k(\varphi_k)$ outside~${\mathcal{K}}_k$.
Therefore~$h_k=\overline\theta_k( g_k) - \overline\theta_k\big(\overline\theta_k(\varphi_k)\big) = 
\overline\theta_k\big(\overline\theta_k(\varphi_k)\big)-\overline\theta_k\big(
\overline\theta_k(\varphi_k)\big)=0$
outside~${\mathcal{K}}_k$. This shows that~$h_k\in C^{\infty}_0(\B_1)$
and it completes the proof of~\eqref{89tTyH}.

Now we observe that~$v^+=0$ a.e. in~$K$ in the trace sense.
This 
and~\eqref{89tTyH}
complete the proof of~\eqref{0iGHBCcv} and so of Lemma~\ref{0iGHBCcvL}.
\end{proof}

While Lemma~\ref{HARM:0p}
gives that the harmonic replacement is ``harmonic'' apart from~$K$,
next result states that it is ``subharmonic'' in the whole of the
domain if the boundary datum is nonnegative:

\begin{lemma}\label{sofar}
If~$\varphi\ge0$, then for every~$\psi\in\H^s_0(\B_1)$ with~$\psi\ge0$
a.e. in~$\B_1$, we have that
$$ \int_{\B_1} |z|^a \nabla \Phi^\varphi_K\cdot\nabla\psi\,dX\le0.$$
\end{lemma}

\begin{proof}
Given~$\eps>0$, we set~$\psi_\eps:=\Phi^\varphi_K-\eps\psi$.
Since~$\Phi^\varphi_K-\varphi\in\H^s_0(\B_1)$
and~$\psi\in\H^s_0(\B_1)$, we have that $\psi_\eps-\varphi\in\H^s_0(\B_1)$.
Furthermore, a.e. in~$K$, we have that~$\psi_\eps=-\eps\psi\le0$
in the trace sense, therefore~$\psi_\eps
\in\tilde\D^\varphi_K $.

{F}rom this
and Lemma~\ref{0iGHBCcvL},
it follows that~$\E(\psi_\eps)-\E(\Phi^\varphi_K)\ge0$
and this gives the desired result.
\end{proof}

For our purposes we will never use Lemma \ref{sofar}, 
but we stated and proved it since it can be a useful consequence 
of the theory developed so far in Section \ref{FHR}. 

\subsection{A monotonicity property for the fractional harmonic replacement}

Now we show that the fractional harmonic replacement
enjoys a monotonicity property with respect to
its boundary data and the constrain:

\begin{theorem}\label{MONOTONE}
Let~$ \H^s(\B_1)\ni\varphi_2\ge\varphi_1\ge0$.
Let also~$K_2\subseteq K_1\subseteq B_{\frac{9}{10}}$
and~$A_1 \subseteq A_2\Subset B_{\frac{9}{10}}$.
Then
$$ \E(\Phi^{\varphi_1}_{ K_1\cup A_1})
-\E(\Phi^{\varphi_1}_{K_1}) \le
\E(\Phi^{\varphi_2}_{K_2\cup A_2})
-\E(\Phi^{\varphi_2}_{ K_2}).$$
\end{theorem}

\begin{proof} We consider the minimization
problem in~$\D^{\varphi_2}_{K_1,\,\Phi^{\varphi_2}_{K_2\cup A_2}}$.
In the notation of Theorem~\ref{MIN-0},
the associated minimizer will be denoted
by~$\Phi^{\varphi_2}_{K_1,\,\Phi^{\varphi_2}_{K_2\cup A_2}}$.

We claim that
\begin{equation}\label{0OKfg67890}
\Phi^{\varphi_1}_{K_1}\le \Phi^{\varphi_2}_{K_1,\,\Phi^{\varphi_2}_{K_2\cup A_2}}.
\end{equation}
To prove this, we let~$g:= \Phi^{\varphi_1}_{K_1}-
\Phi^{\varphi_2}_{K_1,\,\Phi^{\varphi_2}_{K_2\cup A_2}}$
and~$\varphi:= \varphi_1-\varphi_2$. Notice that~$\sup_{\B_1}\varphi\le0$,
thus we can use Lemma~\ref{AUX:0ug}
(with~$c:=0$)
and conclude that~$h:= g^+\in \H^s_0(\B_1)$.
Furthermore, in the trace sense, a.e. in~$K_1$ we have that~$h=
(0-\Phi^{\varphi_2}_{K_1,\,\Phi^{\varphi_2}_{K_2\cup A_2}})^+ \le0$,\
thanks to Lemma~\ref{MPL-2}, and so
\begin{equation}\label{chiamalo}
h\in \D^{0}_{K_1}
\end{equation}
Consequently,
for every~$\delta\in(-1,1)$, we conclude that~$\Phi^{\varphi_2}_{K_1,\,
\Phi^{\varphi_2}_{K_2\cup A_2}}
+\delta h\in
\D^{\varphi_2}_{K_1,\,
\Phi^{\varphi_2}_{K_2\cup A_2}}$ and then, by the
minimizing property of~$
\Phi^{\varphi_2}_{K_1,\,\Phi^{\varphi_2}_{K_2\cup A_2}}
$, it follows that~$\E(
\Phi^{\varphi_2}_{K_1,\,\Phi^{\varphi_2}_{K_2\cup A_2}}
)\le\E(
\Phi^{\varphi_2}_{K_1,\,
\Phi^{\varphi_2}_{K_2\cup A_2}} +\delta h)$.

This implies that
$$ \int_{\B_1} |z|^a \nabla 
\Phi^{\varphi_2}_{K_1,\,
\Phi^{\varphi_2}_{K_2\cup A_2}}
\cdot\nabla h\,dX=0.$$
Hence, we have
\begin{eqnarray*}
\E(h)&=&\int_{\B_1} |z|^a \nabla
(\Phi^{\varphi_1}_{K_1}
-\Phi^{\varphi_2}_{K_1,\,
\Phi^{\varphi_2}_{K_2\cup A_2}}
)\cdot \nabla h\,dX \\
&=& 
\int_{\B_1} |z|^a \nabla
\Phi^{\varphi_1}_{K_1} \cdot \nabla h\,dX.
\end{eqnarray*}
Thus, recalling~\eqref{chiamalo} and~\eqref{0PouyghjYU-1}, we obtain that~$\E(h)=0$.
This, together with Lemma~\ref{CHIA67877:L},
implies that~$h$ vanishes and establishes~\eqref{0OKfg67890}.

Now we set
\begin{equation}\label{9jGHJyhhjJ}
\eta:=
\Phi^{\varphi_2}_{K_1,\,
\Phi^{\varphi_2}_{K_2\cup A_2}} 
-
\Phi^{\varphi_2}_{K_2\cup A_2}
.\end{equation}
Notice that~$ \Phi^{\varphi_2}_{K_1,\,
\Phi^{\varphi_2}_{K_2\cup A_2}} -\varphi_2$ and~$
\Phi^{\varphi_2}_{K_2\cup A_2}-\varphi_2$ belong to~$\H^s_0(\B_1)$,
hence so does~$\eta$. Moreover, a.e. in~$K_1$, in the sense of traces,
we have that~$\eta = \Phi^{\varphi_2}_{K_2\cup A_2}-
\Phi^{\varphi_2}_{K_2\cup A_2}=0$. 
This says that
\begin{equation}\label{0oGjJKK6ghHHJJoop}
\eta\in\D^0_{K_1}
\end{equation}
and so we can use~\eqref{0PouyghjYU-2}
(with~$\psi:=\eta$ here) and conclude that
$$ \E(\Phi^{\varphi_2}_{K_1,\,\Phi^{\varphi_2}_{K_2\cup A_2}}
-\eta)=\E(\Phi^{\varphi_2}_{K_1,\,\Phi^{\varphi_2}_{K_2\cup A_2}}
)+\E(\eta).$$
Thus, from~\eqref{9jGHJyhhjJ},
\begin{equation}\label{7yThHjJjkKKKKl}
\begin{split}
\E(\eta)&\,=\,
\E(\Phi^{\varphi_2}_{K_1,\,\Phi^{\varphi_2}_{K_2\cup A_2}}
-\eta)-\E(\Phi^{\varphi_2}_{K_1,\,
\Phi^{\varphi_2}_{K_2\cup A_2}}
)\\
&\,=\,
\E( \Phi^{\varphi_2}_{K_2\cup A_2} )
-\E(\Phi^{\varphi_2}_{K_1,\,\Phi^{\varphi_2}_{K_2\cup A_2}}
).
\end{split}
\end{equation}
Now, since~$K_1\supseteq K_2$
and~$\Phi^{\varphi_2}_{K_2\cup A_2}=0$ a.e. in~$K_2$, we have
that
$$ \Phi^{\varphi_2}_{K_1,\,\Phi^{\varphi_2}_{K_2\cup A_2}}\in
\D^{\varphi_2}_{K_1,\,\Phi^{\varphi_2}_{K_2\cup A_2}} \subseteq
\D^{\varphi_2}_{K_2,\,\Phi^{\varphi_2}_{K_2\cup A_2}}=
\D^{\varphi_2}_{K_2,\,0}=
\D^{\varphi_2}_{K_2}$$
and so
$$\E(\Phi^{\varphi_2}_{K_1,\,\Phi^{\varphi_2}_{K_2\cup A_2}})\ge
\E(\Phi^{\varphi_2}_{K_2} ),$$
thanks
to the minimality
of~$\Phi^{\varphi_2}_{K_2}$.

This and~\eqref{7yThHjJjkKKKKl}
imply that
\begin{equation}\label{0oIoInDFGHjJ}
\E(\eta)\le
\E( \Phi^{\varphi_2}_{K_2\cup A_2} )
-\E(\Phi^{\varphi_2}_{K_2}).\end{equation}
On the other hand, from Lemma~\ref{0iGHBCcvL},
we know that
$$ \E(\Phi^{\varphi_1}_{K_1\cup A_1})=
\min_{ v\in\tilde\D^{\varphi_1}_{K_1\cup A_1} } \E(v).$$
Therefore, calling~$\psi:=\Phi^{\varphi_1}_{K_1}-v$,
we have that
\begin{equation}\label{PRE-8jHkUgh234FDCVB}
\E(\Phi^{\varphi_1}_{K_1\cup A_1})=
\min_{ \psi\in \Phi^{\varphi_1}_{K_1} -
\tilde\D^{\varphi_1}_{K_1\cup A_1} } \E(\Phi^{\varphi_1}_{K_1}-\psi).
\end{equation}
Now we claim that
\begin{equation}\label{8jHkUgh234FDCVB}
\eta \in \Phi^{\varphi_1}_{K_1} -
\tilde\D^{\varphi_1}_{K_1\cup A_1} .
\end{equation}
For this, we recall~\eqref{9jGHJyhhjJ}, and we have that
$$ \tilde\eta:= \Phi^{\varphi_1}_{K_1} -\eta
=\Phi^{\varphi_2}_{K_2\cup A_2}
-\Phi^{\varphi_2}_{K_1,\,\Phi^{\varphi_2}_{K_2\cup A_2}}
+\Phi^{\varphi_1}_{K_1} .$$
{F}rom this, it follows that~$\tilde\eta-\varphi_1\in\H^s_0(\B_1)$.
Also, a.e. in~$K_1$, we have that~$\tilde\eta
= \Phi^{\varphi_2}_{K_2\cup A_2} - \Phi^{\varphi_2}_{K_2\cup A_2}+0=0$,
in the trace sense. Moreover, a.e. in~$A_1\subseteq A_2$,
we have that~$\tilde\eta=0
-\Phi^{\varphi_2}_{K_1,\,\Phi^{\varphi_2}_{K_2\cup A_2}}
+\Phi^{\varphi_1}_{K_1}\le0$, where~\eqref{0OKfg67890}
has been exploited.
These observations imply that~$\tilde\eta
\in\tilde\D^{\varphi_1}_{K_1\cup A_1}$, which in turn
implies~\eqref{8jHkUgh234FDCVB}.

{F}rom~\eqref{PRE-8jHkUgh234FDCVB}
and~\eqref{8jHkUgh234FDCVB}, we obtain that
\begin{equation}\label{76y54t8pp}
\E(\Phi^{\varphi_1}_{K_1\cup A_1})\le
\E(\Phi^{\varphi_1}_{K_1}-\eta).\end{equation}
Moreover, by~\eqref{0oGjJKK6ghHHJJoop}
and~\eqref{0PouyghjYU-2}
(used here with~$\psi:=\eta$),
we have that
$$
\E(\Phi^{\varphi_1}_{K_1}-\eta)=\E(\Phi^{\varphi_1}_{K_1})+\E(\eta).$$
Thus, formula~\eqref{76y54t8pp}
becomes
$$
\E(\Phi^{\varphi_1}_{K_1\cup A_1})
-\E(\Phi^{\varphi_1}_{K_1})
\le
\E(\Phi^{\varphi_1}_{K_1}-\eta)-
\E(\Phi^{\varphi_1}_{K_1})
= \E(\eta).$$
Therefore, recalling~\eqref{0oIoInDFGHjJ},
$$
\E(\Phi^{\varphi_1}_{K_1\cup A_1})
-\E(\Phi^{\varphi_1}_{K_1})\le
\E( \Phi^{\varphi_2}_{K_2\cup A_2} )
-\E(\Phi^{\varphi_2}_{K_2}).$$
This concludes the
proof of Theorem~\ref{MONOTONE}.
\end{proof}

\section{Energy estimates for the fractional harmonic
replacement}\label{9uTHJkl789}

The goal of this section is to
prove that the energy of
the fractional harmonic replacement in~$K\cup A$
is controlled by
the energy of
the fractional harmonic replacement in~$K$, plus a term
of the order of the $n$-dimensional
measure of the additional set~$A$.
The precise statement of this result goes as follows:

\begin{theorem}\label{THM 1.3}
Let~$\varphi\ge0$ and~$\rho\in [1/4,\,3/4]$.
Let~$K\subseteq \B_1\cap\{z=0\}$ and~$A:= B_\rho\setminus K$. Then
$$ \E(\Phi^\varphi_{K\cup A})-\E(\Phi^\varphi_{K})\le
C\,|A|\, \|\varphi\|_{L^\infty(\B_1)}^2,$$
for some~$C>0$ that depends on~$n$ and~$s$.
\end{theorem}

In the local case of the classical harmonic replacement,
a statement similar to the one in Theorem~\ref{THM 1.3} was
obtained in Lemma~2.3 of~\cite{CSV}. Also,
a fractional case in a different setting was dealt with
in Theorem~1.3 of~\cite{DV} (as a matter of fact,
the right hand side of the estimate obtained
here is more precise than the one
in~\cite{DV} since it only depends on the values
of~$\varphi$ in a fixed ball, and this plays
an important role in the blow-up analysis of the problem).

To proof Theorem~\ref{THM 1.3}, we will reduce
to the radial case. For this,
we will first show that a suitable radial rearrangement decreases
the energy and then estimate the energy in the radial case.
An important step of the proof is also obtained by using
the monotonicity property of Theorem~\ref{MONOTONE},
in order to reduce to the case of constant Dirichlet datum.
The following subsections contain the details of this
strategy.

\subsection{Symmetric rearrangements}\label{S13}

In this subsection, we will consider the symmetric rearrangement in the
variable~$x\in\R^n$, for a fixed~$z\in\R$.
In the forthcoming Theorem~\ref{SR}
we will show that this rearrangement decreases
the energy. 

To this goal, we first state a useful density property
of polynomials in the space we work with.

\begin{lemma}\label{aUYGFfghtyy678}
Let~$v\in \H^s(\B_1)$ and~$\eps>0$.
Then there exists a polynomial~$p_\eps$ such that
$$ \| v-p_\eps\|_{\H^s(\B_1)}\le\eps.$$
\end{lemma}

\begin{proof} By the definition of~$\H^s(\B_1)$ given in
Subsection~\ref{F:S}, we have that there exists~$w_\eps\in C^{\infty}(\B_1)$
such that~$\| v-w_\eps\|_{\H^s(\B_1)}\le\eps$.
Moreover, by the Stone-Weierstra{\ss} Theorem (see e.g.
Lemma~2.1 in~\cite{DSV1}), we have that there exists
a polynomial~$p_\eps$ such that~$\| w_\eps-p_\eps\|_{C^1(\B_1)}\le\eps$.
Therefore
\begin{eqnarray*}
&& \| w_\eps-p_\eps\|_{\H^s(\B_1)}= \sqrt{\int_{\B_1} |z|^a |w_\eps -p_\eps|^2\,dX}+
\sqrt{\int_{\B_1} |z|^a |\nabla w_\eps-\nabla p_\eps|^2\,dX} \\
&&\qquad\le \| w_\eps-p_\eps\|_{C^1(\B_1)} \sqrt{\int_{\B_1} |z|^a \,dX}
\le C\eps,\end{eqnarray*}
for some~$C>0$. As a consequence,
$$ \| v-p_\eps\|_{\H^s(\B_1)}\le
\| v-w_\eps\|_{\H^s(\B_1)}
+\| w_\eps-p_\eps\|_{\H^s(\B_1)}\le \eps +C\,\eps,$$
which implies the desired result after renaming~$\eps$.
\end{proof}

Now, given~$v\in L^\infty(\B_1)$, and fixed any~$z\in\R$,
we consider the Steiner symmetric rearrangement~$v^\sigma(\cdot,z)$
of~$v(\cdot,z)$ (see e.g. Section~2 of~\cite{capriani}).
With this notation, we are ready to establish the main result
of this subsection, that states that the symmetric rearrangement
in the~$x$ variables decreases energy:

\begin{theorem}\label{SR}
For any~$v\in \H^s_0(\B_1)$,
$$ \int_{\B_1} |z|^a |\nabla v^\sigma|^2\,dX
\le \int_{\B_1} |z|^a |\nabla v|^2\,dX.$$
\end{theorem}

\begin{proof} The idea of the proof is to first
prove the desired claim for polynomials using some results
in~\cite{capriani} and then pass to the limit.
The details go as follows.
By Lemma~\ref{aUYGFfghtyy678},
we can take a sequence of polynomials $p_j$ such that
\begin{equation}\label{QQ-2}
\lim_{j\to+\infty} \| v-p_j\|_{\H^s(\B_1)}=0.
\end{equation}
Consequently,
\begin{equation}\label{QQ-3}
\lim_{j\to+\infty}\int_{\B_1} |z|^a |\nabla p_j|^2\,dX
=\int_{\B_1} |z|^a |\nabla v|^2\,dX
.\end{equation}
Now, for any~$(\eta,\zeta)\in\R^n\times\R$,
we set
$$f(\eta,\zeta):=|\eta|^2+|\zeta|^2=|(\eta,\zeta)|^2.$$
Also, for any fixed~$z\in\R$, we set
$$\B_1^z :=\{x\in\R^n {\mbox{ s.t. }}
(x,z)\in\B_1\}. $$
Notice that the Steiner symmetric rearrangement
of~$\B_1^z$ coincides with~$\B_1^z$
itself, thanks to~\eqref{90GHJ}.
By formula~(4.20)
in~\cite{capriani}, we have that
$$ \int_{ \partial^* \{ x\in\B_1^z {\mbox{ s.t. }} p_j^\sigma(x,z)>t\}}
\frac{f(\nabla p_j^\sigma)}{|\nabla_x p_j^\sigma|}\,dx
\le \int_{ \partial^* \{ x\in\B_1^z {\mbox{ s.t. }} p_j(x,z)>t\}}
\frac{f(\nabla p_j)}{|\nabla_x p_j|}\,dx,$$
for any~$t\in\R$, where~$\partial^*$ denotes, as usual,
the reduced boundary in the sense of geometric measure theory.
Thus, by the Coarea Formula,
\begin{eqnarray*}
&& \int_{\B_1^z} |\nabla p_j^\sigma|^2\,dx
= \int_{\B_1^z} f(\nabla p_j^\sigma)\,dx\\
&&\qquad=
\int_\R \left[
\int_{ \partial^* \{ x\in\B_1^z {\mbox{ s.t. }} p_j^\sigma(x,z)>t\}}
\frac{f(\nabla p_j^\sigma)}{|\nabla_x p_j^\sigma|}\,dx\right]\,dt
\\&&\qquad\le \int_\R \left[
\int_{ \partial^* \{ x\in\B_1^z {\mbox{ s.t. }} p_j(x,z)>t\}}
\frac{f(\nabla p_j)}{|\nabla_x p_j|}\,dx \right]\,dt \\
&&\qquad=
\int_{\B_1^z} f(\nabla p_j)\,dx=
\int_{\B_1^z} |\nabla p_j|^2\,dx,
\end{eqnarray*}
for any fixed~$z\in\R$.

Hence, we multiply by~$|z|^a$ and integrate, to obtain
\begin{equation}\label{QQ-1}
\int_{\B_1} |z|^a |\nabla p_j^\sigma|^2\,dX
\le \int_{\B_1} |z|^a |\nabla p_j|^2\,dX.
\end{equation}
Our objective is now to pass to the limit~\eqref{QQ-1}.
The right hand side of~\eqref{QQ-1} will pass to the limit thanks
to~\eqref{QQ-3}, so we discuss now the left hand side.
Since the Schwarz rearrangement is nonexpansive (see e.g. Theorem~3.5
of~\cite{lieb-loss}), we have that, for any fixed~$z\in\R$,
$$ \int_{\B_1^z} |v^\sigma-p_j^\sigma|^2\,dx\le
\int_{\B_1^z} |v-p_j|^2\,dx.$$
So, we multiply by~$|z|^a$ and we integrate over~$z$, and we see that
$$ \int_{\B_1} |z|^a |v^\sigma-p_j^\sigma|^2\,dX\le
\int_{\B_1} |z|^a |v-p_j|^2\,dX.$$
This and \eqref{QQ-2} give that
\begin{equation}\label{QQ-4}
\lim_{j\to+\infty} \int_{\B_1} |z|^a |v^\sigma-p_j^\sigma|^2\,dX=0.
\end{equation}
Now, by~\eqref{QQ-1} and~\eqref{QQ-3}, we have that
$$ \sup_{j\in\N} \int_{\B_1} |z|^a |\nabla p_j^\sigma|^2\,dX <+\infty.$$
Accordingly, by Lemma~\ref{MARS} (see also Theorem 1.30 in \cite{martio}),
we obtain that
$$ \lim_{j\to+\infty}
\int_{\B_1} |z|^a \nabla p_j^\sigma\cdot\nabla\phi\,dX=
\int_{\B_1} |z|^a \nabla v^\sigma\cdot\nabla\phi\,dX,$$
for any~$\phi\in \H^s(\B_1)$. As a consequence,
\begin{eqnarray*}
0 &\le & \liminf_{j\to+\infty}
\int_{\B_1} |z|^a |\nabla p_j^\sigma-\nabla v^\sigma|^2\,dX \\
&=& \liminf_{j\to+\infty}
\int_{\B_1} |z|^a \big( |\nabla p_j^\sigma|^2
+|\nabla v^\sigma|^2 -2\nabla p_j^\sigma\cdot\nabla v^\sigma\big)\,dX
\\ &=& 
\liminf_{j\to+\infty}
\int_{\B_1} |z|^a |\nabla p_j^\sigma|^2\,dX
-\int_{\B_1} |z|^a|\nabla v^\sigma|^2\,dX.
\end{eqnarray*}
This, \eqref{QQ-1} and~\eqref{QQ-3} yield that
\begin{eqnarray*}
&& \int_{\B_1} |z|^a|\nabla v^\sigma|^2\,dX\le \liminf_{j\to+\infty}
\int_{\B_1} |z|^a |\nabla p_j^\sigma|^2\,dX\\
&&\qquad\leq \liminf_{j\to+\infty}
\int_{\B_1} |z|^a |\nabla p_j|^2\,dX=
\int_{\B_1} |z|^a |\nabla v^\sigma|^2\,dX,
\end{eqnarray*}
as desired.
\end{proof}

\subsection{The radial case}

The goal of this subsection is to prove Theorem~\ref{THM 1.3}
in the radial case, that is when
the Dirichlet datum is constant, $K$ is a ball
and~$A$ is a ring. More precisely, we prove that:

\begin{lemma}\label{LEM 4.1}
Let~$\rho\in [1/4,\,3/4]$, $r\in(0,\rho)$ and~$c\in[0,+\infty)$.
Then
$$ \E(\Phi^c_{B_\rho})-\E(\Phi^c_{B_r})\le
C\,c\,|B_\rho\setminus B_r|,$$
for some~$C>0$ that depends on~$n$ and~$s$.
\end{lemma}

\begin{proof} 
If~$c=0$, then~$\Phi^c_{B_\rho}\equiv0$
and~$\Phi^c_{B_r}\equiv0$, in virtue of
Lemmata~\ref{MPL-1} and~\ref{MPL-2}. Thus we may assume that~$c\ne0$.
In fact, by dividing by~$c\ne0$, we may assume that~$c=1$.

We let~$\mu:=\rho-r$ and we observe that
\begin{equation}\label{0-JjkYRggh-1}
\begin{split}
&|B_\rho\setminus B_r| =|B_1|\,(\rho^n-r^n)
=|B_1|\,(\rho-r)\,\sum_{j=1}^{n} \rho^{n-j} r^{j-1}\\
&\qquad \ge |B_1|\,(\rho-r)\,\rho^{n-1}\ge
\frac{|B_1|}{4^{n-1}}\,(\rho-r)=\frac{|B_1|\,\mu}{4^{n-1}}.
\end{split}\end{equation}
Now we fix~$\phi\in C^\infty(\R^{n+1})$ such that~$\phi=1=c$
in~$\R^{n+1}\setminus\B_1$ and~$\phi=0$ in~$B_{3/4}\times\{0\}$.
We let~$C_0:=\E(\phi)$.
By construction~$\phi$ vanishes in~$B_\rho\times\{0\}\supseteq
B_r\times\{0\}$, therefore,
by the minimality properties of~$\Phi^c_{B_\rho}$ and~$\Phi^c_{B_r}$,
we have that
\begin{equation}\label{0-JjkYRggh-2}
\max\{ \E(\Phi^c_{B_\rho}),\, \E(\Phi^c_{B_r})\}
\le \E(\phi) = C_0 = C_0\,c.\end{equation}
We define
\begin{eqnarray*}
&& {\mathcal{C}}_+:=
B_{5/6}\times \left(-\frac12,\frac12\right)\\
{\mbox{and }}&&
{\mathcal{C}}_-:=
B_{4/5}\times \left(-\frac14,\frac14\right).\end{eqnarray*}
By~\eqref{90GHJ},
$$ B_\rho\times\{0\} \Subset
{\mathcal{C}}_- \Subset {\mathcal{C}}_+ \Subset\B_1.$$
Therefore, there exists~$\tau\in C^\infty(\R^{n+1},[0,1])$
such that~$\tau=1$ in~${\mathcal{C}}_-$
and~$\tau=0$ in~$\R^{n+1}\setminus {\mathcal{C}}_+$.

For any~$X\in\R^{n+1}$ we define
$$ \alpha(X):= \left( 1-\frac{\mu}{\rho}\,\tau(X)\right)\,X
=X-\frac{\mu}{\rho}\,\tau(X)\,X.$$
Let also~$1_{n+1}$ be the identity $(n+1)$-dimensional
matrix.
Notice that~$X\mapsto \tau(X)\,X$ is a smooth and compactly
supported function, and so
\begin{equation}\label{78-09-7889-0}
|D\alpha(X)-1_{n+1}|=
\frac{\mu}{\rho}\,\left| D\big( \tau(X)\,X \big) \right|
\le C_1\mu,\end{equation}
for some~$C_1>0$.
Accordingly
\begin{equation}\label{78-09-7889-1}
|\det D\alpha(X)|\ge 1-C_2\mu,\end{equation}
as long as~$\mu$ is small enough.

Now we observe that
\begin{equation}\label{78-09-7889-2-BIS}
\alpha(B_\rho\times\{0\})\subseteq B_r\times\{0\}.
\end{equation}
Indeed, if~$x\in B_\rho$, then~$(x,0)\in {\mathcal{C}}_-$,
thus~$\tau(x,0)=1$, which gives
$$ \alpha(x,0)= \left( 1-\frac{\mu}{\rho}\right)\,(x,0)
=\frac{r}{\rho}\,(x,0),$$
proving~\eqref{78-09-7889-2-BIS}.

We also notice that
\begin{equation}\label{78-09-7889-3}
\alpha(\R^{n+1}\setminus\B_1)\subseteq\R^{n+1}\setminus\B_1.
\end{equation}
Indeed, if~$X\in \R^{n+1}\setminus\B_1$, then
in particular~$X\in
\R^{n+1}\setminus{\mathcal{C}}_+$, which gives that~$\tau(X)=0$
and so~$\alpha(X)=X\in \R^{n+1}\setminus\B_1$,
establishing~\eqref{78-09-7889-3}.

Now we claim that
\begin{equation}\label{78-09-7889-7}
\alpha(\B_1)\subseteq\B_1.
\end{equation}
To prove this, let~$X\in\B_1$. 
If~$X\in
\B_1\setminus{\mathcal{C}}_+$, we have that~$\tau(X)=0$,
thus~$\alpha(X)=X\in\B_1$ and we are done.
If instead~$X\in {\mathcal{C}}_+ = B_{5/6}\times [-1/2,\,1/2]$,
then~$\alpha(X)=\theta(X)\,X$,
for some~$\theta(X)\in[0,1]$, thus~$\alpha(X)$
also lies in~$B_{5/6}\times [-1/2,\,1/2]={\mathcal{C}}_+\subseteq\B_1$,
and this completes the proof of~\eqref{78-09-7889-7}.

Now we observe that
\begin{equation}\label{9idyhGHRDFu}
{\mbox{if $\tilde X=(\tilde x,\tilde z)=\alpha(X)=\alpha(x,z)$,
then }} \frac{|z|}{1+C_3\mu}\le |\tilde z|\le (1+C_3\mu)|z|,
\end{equation}
for some~$C_3>0$, as long as~$\mu$ is sufficiently small.
To prove this, we observe that
$$ \tilde z=
\left( 1-\frac{\mu}{\rho}\,\tau(X)\right)\,z,$$
and this gives~\eqref{9idyhGHRDFu}.

Now we define~$\phi^\star(X):=\Phi^c_{B_r}(\alpha(X))$.
{F}rom~\eqref{78-09-7889-2-BIS}
and~\eqref{78-09-7889-3}, we have that~$\phi^\star\in\D^c_{B_\rho}$,
therefore the minimizing property of~$\Phi^c_{B_\rho}$ gives that
\begin{equation}\label{78-09-7889-4}
\E(\Phi^c_{B_\rho})\le \E(\phi^\star).\end{equation}
On the other hand, by~\eqref{78-09-7889-0},
\eqref{78-09-7889-1}, \eqref{78-09-7889-7}
and~\eqref{9idyhGHRDFu},
\begin{eqnarray*}
\E(\phi^\star) &=& \int_{\B_1}|z|^a \big|\nabla \big(
\Phi^c_{B_r}(\alpha(X))\big)\big|^2\,dX \\
&\le& (1+C_1\mu)^2
\int_{\B_1}|z|^a \big|\nabla \Phi^c_{B_r}(\alpha(X))\big|^2\,dX
\\ &\le& (1+C_4\mu)\,
\int_{\alpha(\B_1)}|\tilde z|^a
\big|\nabla \Phi^c_{B_r}(\tilde X)\big|^2\,d\tilde X\\
&\le& (1+C_4\mu)\,
\int_{\B_1}|\tilde z|^a
\big|\nabla \Phi^c_{B_r}(\tilde X)\big|^2\,d\tilde X\\
&=& (1+C_5\mu)\,\E(\Phi^c_{B_r}),
\end{eqnarray*}
for some~$C_4$, $C_5>0$,
where the change of variable~$\tilde X:=\alpha(X)$ was exploited.

Hence, recalling~\eqref{78-09-7889-4}, we obtain that
$$ \E(\Phi^c_{B_\rho})\le \E(\phi^\star)
\le (1+C_5\mu)\,\E(\Phi^c_{B_r}),$$
provided that~$\mu$ is small enough.
As a consequence, from~\eqref{0-JjkYRggh-1}
and~\eqref{0-JjkYRggh-2},
$$ \E(\Phi^c_{B_\rho})-\E(\Phi^c_{B_r})\le
C_5\mu\E(\Phi^c_{B_r})\le C_6\,|B_\rho\setminus B_r|\,
\E(\Phi^c_{B_r})\le C_7\,c\,|B_\rho\setminus B_r|,$$
for some~$C_6$, $C_7>0$, provided that~$\mu$ is small enough.

This completes the proof of Lemma~\ref{LEM 4.1}
for small~$\mu$, say~$\mu\le\mu_0$ for a suitable~$\mu_0>0$.

Conversely, when~$\mu>\mu_0$, we have that
$$ \E(\Phi^c_{B_\rho})-\E(\Phi^c_{B_r})\le
\E(\Phi^c_{B_\rho})\le C_0\,c \le C_0\,c\,\mu_0^{-1}\,\mu
\le C_8\,\,c\,|B_\rho\setminus B_r|,$$
for some~$C_8>0$,
thanks to~\eqref{0-JjkYRggh-1}
and~\eqref{0-JjkYRggh-2},
which establishes 
Lemma~\ref{LEM 4.1} also when~$\mu>\mu_0$.
\end{proof}

Now we generalize Lemma~\ref{LEM 4.1}
to the case in which the Dirichlet datum is still constant,
but the supporting sets~$K$ and~$A$ are not necessarily
radially symmetric. In this framework, we have:

\begin{lemma}\label{LEM 4.1 - BIS}
Let~$\rho\in [1/4,\,3/4]$ and~$c\in[0,+\infty)$.
Let~$K\subseteq \B_\rho\cap\{z=0\}$ and~$A:= B_\rho\setminus K$. Then
$$ \E(\Phi^c_{K\cup A})-\E(\Phi^c_{K})\le
C\,c\,|A|.$$
for some~$C>0$ that depends on~$n$ and~$s$.
\end{lemma}

\begin{proof} We point out that Lemma~\ref{LEM 4.1 - BIS}
reduces to Lemma~\ref{LEM 4.1} in the special case in which~$K:=B_r$,
with~$r\in(0,\rho)$. In the general case, we argue as follows.
We take~$r$ such that~$|B_r|=|K|$. Then
\begin{equation}\label{0ouhYHUJjJK}
|A|=|B_\rho\setminus K|=|B_\rho|-|K|=|B_\rho|-|B_r|=|B_\rho\setminus B_r|.
\end{equation}
Also, we define~$\psi:=c-\Phi^c_K$.
Notice that~$0\le \psi\le c$, due to Lemmata~\ref{MPL-1} and~\ref{MPL-2}
and~$\psi\in\H^s_0(\B_1)$.
Thus~$\psi\in\D^0_{K,c}$ and so its symmetric rearrangement~$\psi^\sigma$ in the
variable~$x\in\R^n$ (as defined in Subsection~\ref{S13})
satisfies~$\psi^\sigma\in\D^0_{B_r,c}$.

Let~$\psi^\star:=c-\psi^\sigma$. Then~$\psi^\star\in \D^c_{B_r}$,
therefore, by the minimality of~$\Phi^c_{B_r}$, we have that
\begin{equation*}
\E(\Phi^c_{B_r})\le\E(\psi^\star)=\E(\psi^\sigma).
\end{equation*}
On the other hand, by Theorem~\ref{SR}, we know that~$\E(\psi^\sigma)\le
\E(\psi)$. As a consequence
\begin{equation*} 
\E(\Phi^c_{B_r})\le\E(\psi)=\E(\Phi^c_K).\end{equation*}
Now we remark that~$K\cup A=B_\rho$, therefore
$$ \E(\Phi^c_{K\cup A})-\E(\Phi^c_{K})=
\E(\Phi^c_{B_\rho})-\E(\Phi^c_{K})\le
\E(\Phi^c_{B_\rho})-\E(\Phi^c_{B_r}).$$
Then, using  Lemma~\ref{LEM 4.1},
$$ \E(\Phi^c_{K\cup A})-\E(\Phi^c_{K})
\le C\,c\,|B_\rho\setminus B_r|.$$
This and~\eqref{0ouhYHUJjJK}
complete the proof of Lemma~\ref{LEM 4.1 - BIS}.
\end{proof}

\subsection{Completion of the proof of Theorem~\ref{THM 1.3}}

With the arguments introduced till now, we can complete the
proof of Theorem~\ref{THM 1.3}. The idea is that, by the
monotonicity property in Theorem~\ref{MONOTONE},
one can reduce to the case of constant boundary data and then
use Lemma~\ref{LEM 4.1 - BIS}. The details of the proof go
as follows.

\begin{proof}[Proof of Theorem~\ref{THM 1.3}]
We define~$c^\star:= \|\varphi\|_{L^\infty(\B_1)}$,
$K^\star:= K\cap B_\rho$ and
$$ A^\star:= B_\rho\setminus K^\star = B_\rho\setminus(K\cap B_\rho)
=B_\rho\setminus K=A.$$
{F}rom Lemma~\ref{LEM 4.1 - BIS}, we have
\begin{equation}\label{9I79jH}
\E(\Phi^{c^\star}_{K^\star\cup A^\star})-
\E(\Phi^{c^\star}_{K^\star})\le
C\,c^\star\,|A^\star|= C\, \|\varphi\|_{L^\infty(\B_1)}\,|A|.\end{equation}
On the other hand, 
we see that~$c^\star\ge\varphi\ge0$ a.e. in~$\B_2$,
$K^\star\subseteq K\subseteq B_{\frac{9}{10}}$
and~$A^\star =A\Subset B_{\frac{9}{10}}$, therefore,
by Theorem~\ref{MONOTONE},
$$ \E(\Phi^{\varphi}_{ K\cup A})
-\E(\Phi^{\varphi}_{K}) \le
\E(\Phi^{c^\star}_{K^\star\cup A^\star})
-\E(\Phi^{c^\star}_{ K^\star}).$$
Combining this with~\eqref{9I79jH},
we obtain the desired result.
\end{proof}

\section{Density estimates}\label{8ikY89521}

In this section, we deal with density estimates.
A crucial ingredient of our argument will
be the estimate previously obtained in
Theorem~\ref{THM 1.3}.

\subsection{Density estimates from one side}

We start by proving a density estimate from one
side and a uniform bound on the minimizers.

\begin{lemma}\label{L--1}
Assume that~$(u,E)$ is minimizing in~$B_1$, with~$u\ge0$
a.e. in~$\R^n\setminus B_1$ and~$0\in \partial E$.

Then, there exist~$\delta$, $K>0$, 
possibly depending on~$n$, $s$, and~$\sigma$
such that
\begin{equation}\label{CD-3}
|B_{1/2}\setminus E| \ge \delta
\end{equation}
and
\begin{equation}\label{CD-4}
\| u\|_{L^\infty(B_{1/2})}\le K.
\end{equation}
\end{lemma}

\begin{proof} The proof is an appropriate
modification of the one in Lemma~3.1 of~\cite{CSV},
combined with some results in~\cite{DV}.
First we prove~\eqref{CD-3}. For this, for any~$r\in
[1/4,3/4]$, we define
\begin{equation}\label{A678ddd-0}
V_r:=|B_{r}\setminus E| \ {\mbox{ and }} \
a(r):= {\mathcal{H}}^{n-1}\big( (\partial B_r)\setminus E\big).\end{equation}
The terms~$V_r$ and~$a(r)$ play the role of volume and area terms,
respectively.

We suppose, by contradiction, that
\begin{equation}\label{ABCD}
V_{1/2}<\delta\end{equation}
(of course we are free to choose~$\delta$ suitably small, and then
we will obtain a contradiction for such fixed~$\delta$).
We set
\begin{equation}\label{A678ddd}
A:=B_r\setminus E.\end{equation}
By Lemma~3.3 in~\cite{DSV2} we have that
\begin{equation}\label{tyud}
{\mbox{$u\ge0$ a.e.,}}\end{equation} hence
$u\ge 0$ a.e. in~$E$ and~$u=0$ a.e. in~$\R^n\setminus E$.

In particular, $u\ge0$ a.e. in~$E\cup A$ 
and~$u=0$ a.e. in~$(\R^n\setminus E)\setminus A=\R^n\setminus (E\cup A)$.
As a consequence, the pair~$(u, E\cup A)$ is admissible.

Accordingly, from the minimality of~$(u,E)$, we obtain that~${\mathcal{F}}_{B_1}
(u,E)\le {\mathcal{F}}_{B_1}(u, E\cup A)$, that is
\begin{equation}\label{78d}
\Per_\sigma\left(E,B_1\right)-
\Per_\sigma\left(E\cup A,B_1\right)\le 0.
\end{equation}
Also, by~\eqref{oiU67jsdfgH},
\begin{equation}\label{78e}
\Per_\sigma\left(E,B_1\right)-
\Per_\sigma\left(E\cup A,B_1\right)=
L(A,E)- L(A,\R^n\setminus (E\cup A)).\end{equation}
Hence, recalling \eqref{78d}, we conclude that
\begin{equation}\label{6789f}
\begin{split}
L(A,\R^n\setminus A)\,&=L(A,E)+L(A,\R^n\setminus (E\cup A))\\
&=2L(A,\R^n\setminus (E\cup A))+\Per_\sigma\left(E,B_1\right)-
\Per_\sigma\left(E\cup A,B_1\right)\\
&\le 2L(A,\R^n\setminus (E\cup A)) \\
&\le 2L(A,\R^n\setminus B_r).
\end{split}
\end{equation}
Furthermore, using the fractional Sobolev inequality
(see e.g. \cite{DPV12}), we have that
$$\| \chi_A\|_{L^{\frac{2n}{n-\sigma}}(\R^n)}^2
\le C\, \iint_{\R^{2n}}
\frac{|\chi_A(x)-\chi_A(y)|^2}{|x-y|^{n+\sigma}}\, dx\,dy,$$
for some~$C>0$, which may be written as
\begin{equation}\label{6789tatta}
V_r^{\frac{n-\sigma}{n}} \le 2C
\,L(A,\R^n\setminus A).\end{equation}
{F}rom this and~\eqref{6789f}, possibly renaming constants,
we deduce that
\begin{equation}\label{tff}
V_r^{\frac{n-\sigma}{n}} \le C\,
L(A,\R^n\setminus B_r).
\end{equation}
Now, using polar coordinates, we see that, for any~$x\in A\subseteq B_r$,
$$ \int_{\R^n\setminus B_r} \frac{dy}{|x-y|^{n+\sigma}} \le
\int_{\R^n\setminus B_{r-|x|}} \frac{dz}{|z|^{n+\sigma}}\le
C\,\int_{r-|x|}^{+\infty} \rho^{-1-\sigma}\,d\rho\le 
C\,(r-|x|)^{-\sigma},$$
up to renaming constants. Therefore, integrating over~$x\in A=B_r\setminus E$,
we obtain that
\begin{eqnarray*}
&& L(A,\R^n\setminus B_r)\le C\,\int_{B_r\setminus E} (r-|x|)^{-\sigma}\,dx
\\ &&\qquad\le C\,\int_{0}^r a(\rho)\,(r-\rho)^{-\sigma} \rho^{n-1}\,d\rho
\le C\,\int_{0}^r a(\rho)\,(r-\rho)^{-\sigma} \,d\rho.\end{eqnarray*}
So, we plug this into~\eqref{tff} and we conclude that
\begin{equation*}
V_r^{\frac{n-\sigma}{n}} \le C\,
\int_{0}^r a(\rho)\,(r-\rho)^{-\sigma} \,d\rho.\end{equation*}
Now we fix~$t\in[1/4,1/2]$ and we integrate this estimate
in~$r\in[1/4,t]$: we conclude that
\begin{equation}\label{567d8f77}
\int_{1/4}^t V_r^{\frac{n-\sigma}{n}} \,dr
\le C\,\int_0^t \left[ \int_{\rho}^t
a(\rho)\,(r-\rho)^{-\sigma}\,dr \right]\,d\rho
\le C\,t^{1-\sigma}\int_0^t 
a(\rho)\,d\rho \le C\,V_t,
\end{equation}
again up to renaming the constants.
Now we iterate this estimate by setting, for any~$k\ge2$,
$$ t_k:=\frac14+\frac1{2^k} \ {\mbox{ and }} \
v_k:=V_{t_k}.$$
Since~$V_r$ is monotone in~$r$, we have that
$$ \int_{1/4}^{t_k} V_r^{\frac{n-\sigma}{n}} \,dr\ge
\int_{t_{k+1}}^{t_k} V_r^{\frac{n-\sigma}{n}} \,dr
\ge V_{t_{k+1}}^{\frac{n-\sigma}{n}} \,(t_k-t_{k+1})
= 2^{-(k+1)}\,v_{k+1}^{\frac{n-\sigma}{n}} .$$
Hence, if we write~\eqref{567d8f77} with~$t:=t_k$
we obtain that
$$ v_{k+1}^{\frac{n-\sigma}{n}}
\le C^k \,v_k,$$
up to renaming the constants.
Also, by~\eqref{ABCD},
$v_2<\delta$, which is assumed to be conveniently small.
Then, it is easy to see that~$v_k\le C \eta^k$,
for some~$C>0$ and~$\eta\in(0,1)$ (see e.g. formula~(8.18)
in~\cite{DFV-dislocation}) and so
\begin{equation}\label{78909991}
0=\lim_{k\rightarrow+\infty} v_k = V_{1/4}.\end{equation}
As a consequence, $|B_{1/4}\setminus E|=0$, which
is in contradiction with the fact that~$0\in\partial E$
(in the measure theoretic sense) and so it establishes~\eqref{CD-3}.

Now we show the validity of~\eqref{CD-4}. To this scope,
we take~$r=3/4$ in~\eqref{A678ddd}
and we consider the
$s$-harmonic replacement of~$u$ in~$E\cup B_{3/4}=E\cup A$,
according to Definition~1.1
in~\cite{DV}
(notice that the replacement considered in~\cite{DV}
is defined in a setting different than the one introduced
in Section~\ref{FHR} in this paper; as a matter of
fact, the framework introduced in Section~\ref{FHR}
only plays a role in the forthcoming Subsection~\ref{2side}). Namely,
we define~$u_\star$ the function that
minimizes the fractional Dirichlet energy
$$ \iint_{Q_{B_1}}\frac{|v(x)-v(y)|^2}{|x-y|^{n+2s}}
\, dx\, dy$$
among all the functions~$v:\R^n\rightarrow\R$ such that~$v-u\in L^2(\R^n)$,
$v=u$ a.e. in~$\R^n\setminus B_1$ and~$v=0$ a.e. in~$(\R^n\setminus E)\setminus B_{3/4}=\R^n\setminus (E\cup A)$.

The existence (and, as a matter of fact,
uniqueness) of such~$u_\star$ is ensured by Lemma~2.1 of~\cite{DV}.

We set~$\psi:=u_\star-u$. Notice that~$\psi=0$ a.e. in~$(\R^n\setminus B_1) \cup (
\R^n\setminus (E\cup A))$.
Hence, by formula~(2.8) of~\cite{DV} (applied here with~$K:=
\R^n\setminus (E\cup A)$),
\begin{equation}\label{fe7d}
\begin{split}
& \iint_{Q_{B_1}}\frac{|u(x)-u(y)|^2}{|x-y|^{n+2s}}
\, dx\, dy -
\iint_{Q_{B_1}}\frac{|u_\star(x)-u_\star(y)|^2}{|x-y|^{n+2s}}
\, dx\, dy \\ &\qquad\qquad=
\iint_{Q_{B_1}}\frac{|\psi(x)-\psi(y)|^2}{|x-y|^{n+2s}}
\, dx\, dy.
\end{split}\end{equation}
Also
\begin{equation}\label{6f7g}{\mbox{$u_\star\ge 0$ a.e.,}}
\end{equation} thanks to~\eqref{tyud}
and Lemma~2.4 in~\cite{DV}. So,
since~$u_\star=0$ a.e. in~$
\R^n\setminus(E\cup A)$, we see that the
pair~$(u_\star, E\cup A)$ is admissible.

Therefore, by the minimality of~$(u,E)$, we have that
$${\mathcal{F}}_{B_1}(u,E)
\le {\mathcal{F}}_{B_1}(u_\star, E\cup A).$$
This and~\eqref{fe7d} give that
\begin{equation}\label{Ae1}
\begin{split}
& \iint_{Q_{B_1}}\frac{|\psi(x)-\psi(y)|^2}{|x-y|^{n+2s}}
\, dx\, dy \\
=\,&
\iint_{Q_{B_1}}\frac{|u(x)-u(y)|^2}{|x-y|^{n+2s}}
\, dx\, dy -
\iint_{Q_{B_1}}\frac{|u_\star(x)- u_\star(y)|^2}{|x-y|^{n+2s}}
\, dx\, dy \\
=\,& {\mathcal{F}}_{B_1}(u,E)-{\mathcal{F}}_{B_1}(u_\star, E\cup A)+
\Per_\sigma\left(E\cup A,B_1\right)
-\Per_\sigma\left(E,B_1\right) \\
\le \,& \Per_\sigma\left(E\cup A,B_1\right)
-\Per_\sigma\left(E,B_1\right).
\end{split}\end{equation}
Now we recall that~$(-\Delta)^s u_\star=0$ in~$B_{3/4}\subseteq E\cup A$,
due to Lemma~2.3 in~\cite{DV}.
Therefore, recalling~\eqref{6f7g}, we can use
the fractional Harnack inequality (see e.g. Theorem~2.1
in~\cite{kassmann-anewform-CRAS}) and obtain that
\begin{equation}\label{HK}
\sup_{{B_{1/2}}} u_\star \le C \inf_{{B_{1/2}}} u_\star.
\end{equation}
Now we 
claim that
\begin{equation}\label{678ddd}
\|u_\star\|_{L^2(B_{1/2}\setminus E)}^2
\ge c_0 \Big( \sup_{B_{1/2}}u_\star\Big)^2,
\end{equation}
for some~$c_0>0$. To prove this, 
we use~\eqref{HK} to see that
\begin{eqnarray*}
&& \|u_\star\|_{L^2(B_{1/2}\setminus E)}^2 =
\int_{B_{1/2}\setminus E} u_\star^2\,dx
\\ &&\qquad\ge \Big( \inf_{B_{1/2}} u_\star\Big)^2\, |B_{1/2}\setminus E|
\ge \Big( C^{-1} \sup_{{B_{1/2}}} u_\star \Big)^2\, |B_{1/2}\setminus E|
\end{eqnarray*}
and this proves~\eqref{678ddd}, thanks to~\eqref{CD-3}.

Furthermore, since~$\psi=0$ a.e. in~$\R^n\setminus B_1$, we have that
\begin{equation}\label{5678dddpp}\begin{split}
\iint_{\R^{2n}}\frac{|\psi(x)-\psi(y)|^2}{|x-y|^{n+2s}}\,dx\,dy
\,& \ge 
\int_{B_{3/4}} \left[\int_{\R^n\setminus B_1}
\frac{|\psi(x)-\psi(y)|^2}{|x-y|^{n+2s}} dy \right]\,dx
\\ & =
\int_{B_{3/4}} \left[\int_{\R^n\setminus B_1}
\frac{|\psi(x)|^2}{|x-y|^{n+2s}} dy \right]\,dx
\\ & \ge \int_{B_{3/4}} \left[\int_{\R^n\setminus B_{2}}
\frac{|\psi(x)|^2}{|z|^{n+2s}} dz \right]\,dx
\\ &=c \,\|\psi\|_{L^2(B_{3/4})}^2,
\end{split}
\end{equation}
for some~$c>0$.
Then, since~$\psi=u_\star$ in~$B_{1/2}\setminus E$,
we deduce from~\eqref{678ddd} and~\eqref{5678dddpp} that
\begin{equation}\label{Ae0}
\begin{split}
\Big( \sup_{B_{1/2}}u_\star\Big)^2 \,&\le c_0^{-1}
\|\psi\|_{L^2(B_{1/2}\setminus E)}^2 
\\ &\le c_0^{-1}
\|\psi\|_{L^2(B_{3/4})}^2 
\\ &\le C 
\iint_{\R^{2n}}\frac{|\psi(x)-\psi(y)|^2}{|x-y|^{n+2s}}
\, dx\, dy\\ &= C
\iint_{Q_{B_1}}\frac{|\psi(x)-\psi(y)|^2}{|x-y|^{n+2s}}
\, dx\, dy,
\end{split}
\end{equation}
for some~$C>0$.
Now we claim that
\begin{equation}\label{Ae}
{\mbox{$u_\star\ge u$ a.e. in~$\R^n$.}}
\end{equation}
To prove it, we set~$\beta:=(u-u_\star)$
and we remark that
\begin{equation}\label{Ae-0oTYuiHJH-0oYUI}
{\mbox{$\beta^+=0$ a.e. in~$(\R^n\setminus B_1) \cup (\R^n\setminus E)$.}}\end{equation}
Thus, from formula~(2.7) in~\cite{DV}, we have that
\begin{equation}\label{Ae-0oTYuiHJH}
\iint_{Q_{B_1}}\frac{
\big( u_\star(x)-u_\star(y)\big)\,\big( \beta^+(x)-\beta^+(y)\big)
}{|x-y|^{n+2s}}
\, dx\, dy=0.
\end{equation}
Moreover, fixed~$\eps\in(0,1)$,
we define~$u_\eps:= u -\eps\beta^+$. We notice that 
\begin{equation}\label{Ae-0oTYuiHJH9iyg}
{\mbox{$u_\eps\ge0$ a.e. in~$E$.}}\end{equation}
Indeed, let~$x\in E$: if~$\beta^+(x)=0$ then~$u_\eps(x)=u(x)\ge0$
(up to negligible sets); if instead~$\beta^+(x)>0$, then~$\beta^+(x)=
u(x)-u_\star(x)$, thus~$u_\eps(x)= (1-\eps)u(x)+\eps u_\star(x)\ge0$,
thanks to~\eqref{6f7g}.
This proves~\eqref{Ae-0oTYuiHJH9iyg}.

{F}rom~\eqref{Ae-0oTYuiHJH-0oYUI}
and~\eqref{Ae-0oTYuiHJH9iyg}, we obtain that~$(u_\eps,E)$ is an admissible
competitor for~$(u,E)$, therefore, by the minimality of~$(u,E)$,
we see that
$$ \iint_{Q_{B_1}} 
\frac{
\big( u(x)-u(y)\big)\,\big( \beta^+(x)-\beta^+(y)\big)
}{|x-y|^{n+2s}}
\, dx\, dy\le0.$$
This and~\eqref{Ae-0oTYuiHJH} give that
$$\iint_{Q_{B_1}}\frac{
\big( \beta(x)-\beta(y)\big)\,\big( \beta^+(x)-\beta^+(y)\big)
}{|x-y|^{n+2s}}
\, dx\, dy\le0.$$
On the other hand (see e.g. formula~(8.10)
in~\cite{DFV-dislocation}), we have that
$$ \big( \beta(x)-\beta(y)\big)\,\big( \beta^+(x)-\beta^+(y)\big)
\ge \big|\beta^+(x)-\beta^+(y)\big|^2,$$ so we deduce that
$$\iint_{Q_{B_1}}\frac{
\big|\beta^+(x)-\beta^+(y)\big|^2
}{|x-y|^{n+2s}}
\, dx\, dy\le0.$$
This says that~$\beta^+=0$ a.e. in~$\R^n$,
which in turn implies~\eqref{Ae}.

{F}rom~\eqref{Ae0} and~\eqref{Ae} we obtain that
$$ \Big( \sup_{B_{1/2}} u\Big)^2 \le C
\iint_{Q_{B_1}}\frac{|\psi(x)-\psi(y)|^2}{|x-y|^{n+2s}}
\, dx\, dy.
$$
By plugging this into~\eqref{Ae1}
we conclude that
$$ \Big( \sup_{B_{1/2}} u\Big)^2
\le C\,\Big( \Per_\sigma\left(E\cup A,B_1\right)
-\Per_\sigma\left(E,B_1\right) \Big).$$
Hence, recalling~\eqref{78e}, we deduce that
\begin{eqnarray*}
\Big( \sup_{B_{1/2}} u\Big)^2 &\le&
C\,\Big(L(A,
\R^n\setminus (E\cup A)
)-L(A,E)\Big)\\
&\le& C\,L(A,
\R^n\setminus (E\cup A)
)\\
&\le& C\,L(B_{3/4},\R^n\setminus B_{3/4})\\
&\le& C,\end{eqnarray*}
up to relabeling the constants. This completes the proof of~\eqref{CD-4}.
\end{proof} 

Now, recalling the definition in \eqref{EXTENDED}, 
we prove a uniform bound also for the extension function of minimizers. 
This will play a crucial role in the proof of Lemma \ref{L--1bis}, 
in order to obtain that the constant $\delta$ does not depend 
on the quantity $\Lambda$ (see formula \eqref{LAMBDA:FIN} below). 
Indeed, differently from the ``local'' case (see \cite{CSV}), 
the energy estimate for the fractional harmonic replacement 
provided by Theorem \ref{THM 1.3} depends on the $L^{\infty}$-norm 
of the extension function, and this makes the blow-up analysis more delicate. 

More precisely, we have:

\begin{lemma}\label{L-ext}
Let~$(u,E)$ be a minimizing pair in~$B_1$, with~$u\ge0$
a.e. in~$\R^n\setminus B_1$ and~$0\in \partial E$. 
Suppose that 
\begin{equation}\label{LAMBDA:FIN}
\int_{\R^n} \frac{|u(y)|}{1+|y|^{n+2s}}\,dy\le\Lambda, 
\end{equation}
for some $\Lambda>0$. Let also $\overline{u}$ be the as in \eqref{EXTENDED}. 

Then, there exists~$K>0$, possibly depending on~$n$, $s$, and~$\sigma$
such that
\begin{equation}\label{CD-BIS}
\| \overline u \|_{L^\infty(\B_{5/9})}\le K. 
\end{equation}
\end{lemma}

\begin{proof}
The proof of \eqref{CD-BIS} is a suitable variation of the one of \eqref{CD-4}. 
The difference here is that we will consider the harmonic replacement 
that we constructed in Section \ref{FHR}. 
For this, we consider the extension function $\overline u$ 
of $u$, as defined in \eqref{EXTENDED}. 

Moreover, we set 
\begin{equation}\label{def A}
A:=B_{3/4}\setminus E 
\end{equation}
and we  observe that $E\cup B_{3/4}=E\cup A$. 
We consider the fractional harmonic replacement of~$\overline u$, as introduced in Section \ref{FHR}, 
by prescribing $B_{9/10}\setminus(E\cup A)$ as vanishing set. 
More precisely, with the notation of Theorem \ref{MIN-0}, 
we consider $\Phi^{\overline u}_{B_{9/10}\setminus (E\cup A)}$. For shortness of notation, we define 
\begin{equation}\label{qw29}
\tilde w:=\Phi^{\overline u}_{B_{9/10}\setminus (E\cup A)}.\end{equation}
Notice that $\tilde w=\overline u$ in $\B_2\setminus\B_1$, therefore (up to extending $\tilde w$ 
outside $\B_2$) we can say that 
\begin{equation}\label{qw30}
\tilde w = \overline u \quad {\mbox{ in }} \R^{n+1}\setminus \B_1.
\end{equation}
Notice also that 
\begin{equation}\label{qw31}
\tilde w \ge 0 \quad {\mbox{ in }}\R^{n+1}. 
\end{equation}
Indeed, $u\ge 0$ thanks to Lemma 3.3 in \cite{DSV2}, 
and so $\overline u\ge0$, in view of \eqref{EXTENDED}. 
Therefore \eqref{qw31} follows from Lemma \ref{MPL-2}. 

Now we set 
\begin{equation}\label{basta}
F:=E\cup A
\end{equation}
and we claim that 
\begin{equation}\label{qw32}
\tilde w(x,0)\ge 0 {\mbox{ a.e. }}x\in F, {\mbox{ and }} \tilde w(x,0)\le0 {\mbox{ a.e. }} x\in\R^n\setminus F. 
\end{equation}
For this, we first observe that we only need to prove that $\tilde w(x,0)=0$ a.e. $x\in\R^n\setminus F$, 
thanks to \eqref{qw31}. Now, if $x\in B_{9/10}\setminus F$ then $\tilde w(x,0)=0$ 
by the definition of harmonic replacement in \eqref{qw29}. 
Moreover, if $x\in \big(\R^n\setminus B_{9/10}\big)\setminus F$ then $(x,0)\in\R^{n+1}\setminus\B_1$ 
(recall \eqref{90GHJ}), and so, by \eqref{qw30}, we have that 
$\tilde w(x,0)=\overline u(x,0) =u(x)=0$, since $x\in\R^n\setminus E$. 
This shows \eqref{qw32}. 

Now we define~${\mathcal{U}}:=\B_{\frac{11}{10}}$ and we observe that 
$$ {\mathcal{U}}\cap \{z=0\} = B_{\frac{99}{100}}\times \{0\} \subset B_1\times \{0\},$$
thanks to \eqref{90GHJ}. This, \eqref{basta} and \eqref{qw32} imply that the assumptions 
of Lemma \ref{EXT-L} are satisfied (with $r:=1$), and so 
$$ \int_{{\mathcal{U}}} |z|^a|\nabla \overline u|^2\,dX + \Per_\sigma(E,B_1) 
\le \int_{{\mathcal{U}}} |z|^a|\nabla \tilde w|^2\,dX + \Per_\sigma(F,B_1).$$
Recalling \eqref{qw30}, we can rewrite the formula above as 
\begin{equation}\label{qw45}
\int_{\B_1} |z|^a|\nabla \overline u|^2\,dX + \Per_\sigma(E,B_1) 
\le \int_{\B_1} |z|^a|\nabla \tilde w|^2\,dX + \Per_\sigma(F,B_1).\end{equation}
Now we observe that, by \eqref{oiU67jsdfgH}, 
$$ \Per_\sigma\left(E\cup A,B_1\right)-\Per_\sigma\left(E,B_1\right)=
L(A,\R^n\setminus (E\cup A))-L(A,E), $$ 
therefore, recalling \eqref{basta}, we have that 
\begin{eqnarray*}
&& \Per_\sigma\left(F,B_1\right)-\Per_\sigma\left(E,B_1\right)=
L(A,\R^n\setminus (E\cup A))-L(A,E)\\
&&\qquad \le L(A,\R^n\setminus (E\cup A)) \le L(B_{3/4}, \R^n\setminus B_{3/4}) \le C,
\end{eqnarray*}
for some $C>0$. From this and \eqref{qw45}, we conclude that 
\begin{equation}\label{qw46}
\int_{\B_1} |z|^a|\nabla \overline u|^2\,dX - \int_{\B_1} |z|^a|\nabla \tilde w|^2\,dX\le C.  
\end{equation}

Now we set $\Psi:=\tilde w-\overline u$. Notice that 
\begin{equation}\label{we32}
{\mbox{$\Psi=0$ in $\R^{n+1}\setminus\B_1$,}}\end{equation} 
thanks to \eqref{qw30}. Moreover, if $x\in B_{9/10}\setminus (E\cup A)$ then (in the sense of traces)
$$ \Psi(x,0)=\tilde w(x,0) -\overline u(x,0) = -\overline u(x,0) =-u(x)=0,$$
since $x\in \R^{n}\setminus E$ (recall also the definition of fractional harmonic replacement
in \eqref{qw29}). Hence\footnote{The careful reader may have noticed that $K$
here recalls the set notation in the fractional harmonic
replacement framework and of course cannot be confused
with the constant~$K$ in \eqref{CD-BIS}.}, 
$\Psi\in\mathcal{D}^0_{K}$, where $K:=B_{9/10}\setminus (E\cup A)$. 
Therefore, \eqref{0PouyghjYU-2} gives that 
$$ \int_{\B_1}|z|^a|\nabla \overline u|^2\,dX = 
\int_{\B_1}|z|^a|\nabla \tilde w|^2\,dX + \int_{\B_1}|z|^a|\nabla \Psi|^2\,dX.$$ 
Plugging this information into \eqref{qw46}, we obtain
\begin{equation}\label{uytr65}
\int_{\B_1}|z|^a|\nabla \Psi|^2\,dX\le C.
\end{equation}

Also, from \eqref{we32} and Lemma \ref{EVE} (recall that $\overline u$ is even in $z$ by definition), 
we have 
$$ \int_{\B_1}|z|^a|\nabla \Psi|^2\,dX = \int_{\R^{n+1}}|z|^a|\nabla \Psi|^2\,dX
= 2 \int_{\R^{n+1}_+}|z|^a|\nabla \Psi|^2\,dX,$$
where $\R^{n+1}_+:=\R^{n+1}\cap \{z>0\}$. So we set $2^*_s:= \frac{2n}{n-2s}$ and we
use Proposition 1.2.1 in \cite{maria}, obtaining that 
\begin{equation}\begin{split}\label{we129}
& \int_{\B_1}|z|^a|\nabla \Psi|^2\,dX \ge C\left(\int_{\R^n}|\Psi(x,0)|^{2^*_s}\,dx\right)^{2/2^*_s}\\
&\quad \ge C\left(\int_{B_{1/2}\setminus E}|\Psi(x,0)|^{2^*_s}\,dx\right)^{2/2^*_s}
 \ge |B_{1/2}\setminus E|^{ \frac{2}{2^*_s}-1 }
\int_{B_{1/2}\setminus E}|\Psi(x,0)|^{2}\,dx \\ &\quad
\ge | B_{1/2}\setminus E |^{ \frac{2}{2^*_s} }
\left( \inf_{B_{1/2}\setminus E}|\Psi|\right)^2
\ge C\left( \inf_{B_{1/2}\setminus E}|\Psi|\right)^2,
\end{split}\end{equation}
for some $C>0$, thanks to \eqref{CD-3} 
(notice that the H\"older inequality was also used). 

Now we notice that, if $x\in B_{1/2}\setminus E$, then $\Psi(x,0)=\tilde w(x,0)-\overline u(x,0)=
\tilde w(x,0)-u(x)=\tilde w(x,0)$ in the sense of traces. 
Using this information into \eqref{we129} we get 
\begin{equation}\label{wer34}
\int_{\B_1}|z|^a|\nabla \Psi|^2\,dX \ge C \left( \inf_{B_{1/2}\setminus E}|\tilde w|\right)^2.
\end{equation}

Now we observe that $\B_{5/9}$ is compactly contained in $\B_1\setminus K$, 
where $K= B_{9/10}\setminus (E\cup A)$. Indeed, by \eqref{90GHJ} and \eqref{def A}, 
\begin{eqnarray*}
&& \B_{5/9}= B_{\frac{1}{2}} \times \left(-\frac{5}{9},\frac{5}{9}\right) \Subset
B_{\frac{3}{4}} \times \left(-1,1\right)\\
&&\qquad \subset \left(\big(E\cup A\big)\cap B_{9/10}\right) \times (-1,1)
\subset \B_1\setminus K.\end{eqnarray*}
Also, Lemma \ref{HARM:0p} says that 
$$ {\rm div}\, (|z|^a \nabla\tilde w)=0 $$
in the interior of~$\B_1\setminus K$, in the distributional sense. 
Therefore (recalling also \eqref{qw31}) we can use the Harnack inequality (see e.g. Theorem 2.3.8 
in \cite{FKS}) and we obtain that 
$$ \sup_{\B_{5/9}} \tilde w\le C\inf_{\B_{5/9}}\tilde w,$$
for some constant $C>0$ independent of $\tilde w$. 
This, together with \eqref{wer34} and \eqref{qw31}, gives that 
$$ \int_{\B_1}|z|^a|\nabla \Psi|^2\,dX \ge  C \left( \inf_{B_{1/2}\setminus E}\tilde w\right)^2
\ge C \left( \inf_{\B_{5/9}}\tilde w\right)^2
\ge C \left( \sup_{\B_{5/9}}\tilde w\right)^2, 
$$ 
up to relabelling $C$. From this and \eqref{uytr65} we obtain that 
\begin{equation}\label{uytr65-1}
\sup_{\B_{5/9}}\tilde w \le C,
\end{equation}
for some $C>0$, depending only on $n$, $s$ and $\sigma$. 

Now we claim that 
\begin{equation}\label{uytr65-2}
\tilde w\ge \overline u \quad {\mbox{ in }} \R^{n+1}.
\end{equation}
To prove this, we set $\beta:=\overline u-\tilde w$, 
and we observe that $\beta^+=0$ in $\R^{n+1}\setminus\B_1$, due to \eqref{qw30}. 
Furthermore, if $x\in \R^n\setminus \left(B_{9/10}\setminus (E\cup A)\right)$, 
then $\beta^+(x,0)=\left( \overline u(x,0)-\tilde w(x,0)\right)^+ =0$ in the sense of traces. 
Therefore, by \eqref{0PouyghjYU-1},
\begin{equation}\label{gfdtr43}
\int_{\B_1}|z|^a \nabla\tilde w \cdot\nabla\beta^+\,dX=0.
\end{equation}
Now we fix $\epsilon\in(0,1)$ and we define $\overline u_\epsilon:=\overline u-\epsilon\beta^+$. 
Notice that 
\begin{equation}\label{90op}
\overline u_\epsilon =\overline u \quad {\mbox{ in }}\R^{n+1}\setminus \B_1. 
\end{equation}
Also, we observe that 
\begin{equation}\label{pos amm}
\overline u_\epsilon\ge 0.
\end{equation}
Indeed, if $\beta^+=0$ then $\overline u_\epsilon =\overline u\ge 0$; 
if instead $\beta^+>0$ then $\overline u_\epsilon =(1-\epsilon)\overline u +\tilde w\ge 0$. 
This proves \eqref{pos amm}. 

Moreover, if $x\in\R^n\setminus E$, 
$$ \overline u_\epsilon(x,0) = \overline u(x,0) -\epsilon\left(\overline u(x,0)-\tilde w(x,0)\right)^+ 
= -\epsilon \left(-\tilde w(x,0)\right)^+ =0,$$ 
thanks to \eqref{qw31}. This, \eqref{90op} and \eqref{pos amm} 
imply that the assumptions of Lemma \ref{EXT-L} are satisfied with $\mathcal{U}:=\B_{\frac{11}{20}}$, 
and so 
$$ \int_{\mathcal{U}}|z|^a|\nabla \overline u|^2\,dX\le 
\int_{\mathcal{U}}|z|^a|\nabla \overline u_\epsilon|^2\,dX. $$
Using \eqref{90op}, we can write 
$$ \int_{\B_1}|z|^a|\nabla \overline u|^2\,dX\le 
\int_{\B_1}|z|^a|\nabla \overline u_\epsilon|^2\,dX, $$
which, recalling the definition of $\overline u_\epsilon$, implies that
$$ \int_{\B_1}|z|^a\nabla \overline u\cdot\nabla\beta^+\,dX\le 0. $$
This and \eqref{gfdtr43} give that 
$$ \int_{\B_1}|z|^a|\nabla\beta^+|^2\,dX=
\int_{\B_1}|z|^a\nabla \beta\cdot\nabla\beta^+\,dX\le 0. $$
Therefore, we have that $\beta^+=0$ in $\B_1$. This and \eqref{qw30} imply \eqref{uytr65-2}. 

From \eqref{uytr65-1} and \eqref{uytr65-2} we obtain that 
$$ \sup_{\B_{5/9}}\overline u \le C,$$
for some $C>0$, possibly depending on $n$, $s$ and $\sigma$. This shows \eqref{CD-BIS}. 
\end{proof}

We remark that the finiteness of~$\Lambda$ in~\eqref{LAMBDA:FIN} 
is only used in order to have a well-defined extended function~$\overline u$. 
In particular, the quantity~$K$ in Lemma \ref{L-ext} does not depend on~$\Lambda$.

\subsection{Density estimates from the other side}\label{2side}

In Lemma~\ref{L--1}, a density estimate from one side
was obtained, namely we proved that the complement of~$E$
has positive density near the free boundary.

The purpose of this subsection is to prove that also the set~$E$
has positive density near the free boundary.

To this goal, we need to modify appropriately the argument in
Lemma~\ref{L--1}, by using
the machinery developed in the previous sections.
With respect to the argument developed in
the proof of Lemma~\ref{L--1}, in this subsection
the sets in~\eqref{A678ddd-0}
and~\eqref{A678ddd} are replaced by the similar
quantities in which the intersection with~$E$
(rather than with the complement of~$E$)
is taken into account (see \eqref{0oGJKio0}
and~\eqref{0oGJKio0-bis}
below).

This apparently minor difference causes
a conceptual difficulty
in terms of harmonic replacements: indeed, in the proof
of Lemma~\ref{L--1}, the competitor was built
by extending the positivity set of the minimizer~$u$,
while, in the case considered here, the positivity set gets reduced
in the competitor, i.e. the competitor is forced to attain zero value
on a larger set, and this makes its Dirichlet energy
possibly bigger. For this reason, one needs to
estimate the error in the Dirichlet energy
produced by this further constrain on the zero set.
This is the point in which Theorem~\ref{THM 1.3} and Lemma~\ref{L-ext} 
come into play. Indeed, for this estimate, we need
to control the energy difference with a term only
involving the measure of the additional zero set
and the local size of the data
(this is the reason for introducing the fractional
harmonic replacement in the extended variables
in Section~\ref{FHR}
and for considering the extended problem in Lemma~\ref{EXT-L}).

\begin{lemma}\label{L--1bis} Let~$(u,E)$ be a minimizing pair in~$B_2$, 
with~$u\ge0$ a.e. in~$\R^n\setminus B_2$. Suppose that 
$$ \int_{\R^n} \frac{|u(y)|}{1+|y|^{n+2s}}\,dy \le \Lambda,$$
for some~$\Lambda>0$.

Assume also that~$0\in\partial E$.

Then, there exists~$\delta>0$, possibly depending on~$n$, $s$ and~$\sigma$ such that
$$ |B_{1/2}\cap E|\ge\delta.$$
\end{lemma}

\begin{proof}
First of all, we notice that~$u\ge0$ in the whole of~$\R^n$,
thanks to Lemma~3.3 of~\cite{DSV2}. Also,
for any~$r\in [1/4,3/4]$, we define
\begin{equation}\label{0oGJKio0}
V_r:=|B_r\cap E| \;{\mbox{ and }}\;
a(r):={\mathcal{H}}^{n-1}\big((\partial B_r)\cap E\big).\end{equation}
The desired result will follow by arguing by contradiction.
Suppose that the desired result does not hold.
Then~$V_{1/2}<\delta$. We will find a contradiction by taking~$\delta$
conveniently small. To this goal, we set
\begin{equation}\label{0oGJKio0-bis}
A:=B_r\cap E.\end{equation}
We let~$\overline u:\R^{n+1}\to\R$ be the extension of~$u$,
according to \eqref{EXTENDED}.

We consider the fractional harmonic replacement of~$\overline u$, as introduced
in Section~\ref{FHR}, prescribing~$(B_{9/10}\setminus E)\cup A$
as supporting sets, i.e., in the notation
of Theorem~\ref{MIN-0}, we
consider~$\Phi^{\overline u}_{(B_{9/10}\setminus E)\cup A}$,
and we define, for short,
\begin{equation}\label{89UhgfdFG}
\tilde v := \Phi^{\overline u}_{(B_{9/10}\setminus E)\cup A}
.\end{equation}
Notice that~$\tilde v =\overline u$ in~$\B_2\setminus\B_1$, so up to extending~$\tilde v $
outside~$\B_2$, we can write
\begin{equation}\label{98yhFGhJKl09}
{\mbox{$\tilde v =\overline u$ in~$\R^{n+1}\setminus\B_1$.}}
\end{equation}
We also set
\begin{equation}\label{GH HJK}
F:=E\setminus A.\end{equation}
We notice that
\begin{equation}\label{GH HJK-8902}
{\mbox{$\tilde v (x,0)\ge0$ a.e. $x\in F$, and
$\tilde v (x,0)\le0$ a.e. $x\in \R^n\setminus F$.}}\end{equation}
Indeed, $u\ge0$ due to Lemma~3.3 in~\cite{DSV2},
hence~$\overline u\ge0$, in view of~\eqref{EXTENDED}.
Therefore~$\tilde v \ge0$, thanks to Lemma~\ref{MPL-2}.
As a consequence~$\tilde v (x,0)\ge0$ 
in the trace sense.
So it only remains to prove that~$\tilde v (x,0)=0$ a.e. $x\in \R^n\setminus F$.
For this, notice that
$$ \R^n\setminus F =
\R^n\setminus (E\setminus A)
=(\R^n\setminus E)\cup A.$$
So, if~$x\in (B_{9/10}\setminus E)\cup A$,
we have that~$\tilde v (x,0)=0$
by definition of fractional replacement. Also, if~$x\in (\R^n\setminus B_{9/10})\setminus E$,
then~$(x,0)\in \R^{n+1}\setminus \B_1$, due to~\eqref{90GHJ},
and so, by~\eqref{98yhFGhJKl09},
in this case we have~$\tilde v (x,0)=\overline u(x,0)=u(x)=0$,
since~$x\in\R^n\setminus E$.
This proves~\eqref{GH HJK-8902}.

Now we define~${\mathcal{U}}:= \B_{\frac{11}{10}}$
and we observe that
$$ {\mathcal{U}}\cap \{z=0\} = B_{\frac{99}{100}}\times\{0\}\subset
B_1\times\{0\},$$
due to~\eqref{90GHJ}. {F}rom this, \eqref{GH HJK} and~\eqref{GH HJK-8902},
we see that the assumptions of Lemma~\ref{EXT-L}
are satisfied (with~$r:=1$): so we obtain that
$$ \int_{\mathcal{U}} |z|^a |\nabla \overline u|^2\,dX
+\Per_\sigma(E,B_1)\le
\int_{\mathcal{U}} |z|^a |\nabla \tilde v|^2\,dX
+\Per_\sigma(F,B_1).$$
Thus, using again~\eqref{98yhFGhJKl09},
\begin{equation}\label{7uhn4rfhj5678} \int_{\B_1} |z|^a |\nabla \overline u|^2\,dX
+\Per_\sigma(E,B_1)\le
\int_{\B_1} |z|^a |\nabla \tilde v|^2\,dX
+\Per_\sigma(F,B_1).
\end{equation}
Now, by~\eqref{GH HJK},
\begin{eqnarray*}
&& \Per_\sigma(F,B_1)-
\Per_\sigma(E,B_1)=
\Per_\sigma(E\setminus A,B_1)-
\Per_\sigma(E,B_1)\\ &&\qquad=L(A, E\setminus A)-
L(A,\R^n\setminus E).\end{eqnarray*}
By inserting this information into~\eqref{7uhn4rfhj5678}
and recalling~\eqref{89UhgfdFG}
we obtain that
\begin{equation}
\label{09876678uyf}\begin{split}
&L(A,\R^n\setminus E)-L(A, E\setminus A)\le
\int_{\B_1}|z|^a |\nabla \tilde v|^2\,dX
-\int_{\B_1}|z|^a |\nabla \overline u|^2\,dX
\\&\qquad =
\E(\Phi^{\overline u}_{(B_{9/10}\setminus E)\cup A})-
\E(\overline u).\end{split}\end{equation}
On the other hand, $\overline u$ vanishes on~$(B_{9/10}\setminus E)\times\{0\}$,
thus, by the minimality of~$\Phi^{\overline u}_{B_{9/10}\setminus E}$,
we have that
$$ \E(\Phi^{\overline u}_{B_{9/10}\setminus E})\le\E(\overline u).$$
Using this inequality into~\eqref{09876678uyf} and recalling
Theorem~\ref{THM 1.3}, we obtain
\begin{equation}\label{9YTHJ}
\begin{split}
L(A,\R^n\setminus E)-L(A, E\setminus A)
\,&\le
\E(\Phi^{\overline u}_{(B_{9/10}\setminus E)\cup A})
- \E(\Phi^{\overline u}_{B_{9/10}\setminus E})
\\ &\le
C\,|A|\, \|\overline u\|_{L^\infty(\B_1)}^2.\end{split}\end{equation}
{F}rom Lemma~\ref{L-ext}, we have a uniform bound on $\| \overline u\|_{L^\infty(\B_1)}$.
This and~\eqref{9YTHJ} give
$$ L(A,\R^n\setminus E)-L(A, E\setminus A)
\le C\,|A|.$$
Then, the argument in~\cite{CSV} can be repeated verbatim
(see in particular from the first formula in display after~(3.2)
to the fifth line below~(3.4)) and one obtains that~$V_{1/4}=0$.
This contradicts the fact that~$0\in\partial E$ and so it completes
the proof of
Lemma~\ref{L--1bis}.
\end{proof}

By putting together Lemmata~\ref{L--1} and~\ref{L--1bis}
we obtain:

\begin{corollary}\label{67dfffsjjsj}
Assume that~$(u,E)$ is minimizing in~$B_1$, with~$u\ge0$
a.e. in~$\R^n\setminus B_1$ and~$0\in \partial E$. 
Suppose that 
$$ \int_{\R^n} \frac{|u(y)|}{1+|y|^{n+2s}}\,dy \le \Lambda,$$
for some~$\Lambda>0$.

Then, there exist~$\delta$, $K>0$,
possibly depending on~$n$, $s$ and~$\sigma$
such that
\begin{equation}\label{CD-3-XX}
\min\Big\{ |B_{1/2}\setminus E|, \, |B_{1/2}\cap E|\Big\} \ge \delta
\end{equation}
and
\begin{equation}\label{CD-4-XX}
\| u\|_{L^\infty(B_{1/2})}\le K.
\end{equation}
\end{corollary}

We remark that the quantity~$K$ appearing in~\eqref{CD-4-XX}
does not depend on~$\Lambda$.
This fact will allow us to rescale the picture and
deduce from~\eqref{CD-4-XX} 
a universal growth from the
free boundary, as stated in the following result:

\begin{corollary}\label{67dfffsjjsj-BIS}
Assume that~$(u,E)$ is minimizing in~$B_6$, with~$u\ge0$
a.e. in~$\R^n\setminus B_6$ and~$0\in \partial E$. 
Suppose that 
$$ \int_{\R^n} \frac{|u(y)|}{1+|y|^{n+2s}}\,dy \le \Lambda,$$
for some~$\Lambda>0$.

Also, for any $x\in B_{1}\cap E$, we define $d(x):=\dist (x,\partial E)$ 
to be the distance of a point $x$ from the free boundary. 

Then, there exists~$K>0$,
possibly depending on~$n$, $s$ and~$\sigma$
such that
\begin{equation}\label{CD-10-XX-BIS}
|u(x)|\le K (d(x))^{s-\frac{\sigma}{2}} \quad {\mbox{ for any }} x\in B_{1}\cap E.
\end{equation}
Furthermore, 
\begin{equation}\label{CD-11-XX-BIS}
|u(x)|\le K|x|^{s-\frac{\sigma}{2}} \quad {\mbox{ for any }} x\in B_{1}\cap E.
\end{equation}
\end{corollary}

\begin{proof}
We fix $x_0\in B_{1}\cap E$ and we set $r_0:=d(x_0)$. Let also $p_0\in\partial E\cap \partial B_{r_0}(x_0)$ 
be such that $r_0=d(x_0)=|x_0-p_0|$. 
Notice that $r_0\le1$, since $0\in\partial E$, and so $|p_0|\le2$. 

Now, we define
\begin{equation}\label{rescaled-BIS}
u_{r_0}(x):=r_0^{\frac{\sigma}{2}-s} u(r_0x+p_0)\;{\mbox{ and }}\;
E_{r_0}:= \frac{1}{r_0}\left(E -p_0\right),
\end{equation}
and we observe that~$(u_{r_0},E_{r_0})$ is a minimizing pair in $B_{6/r_0}(p_0/r_0)$. 
Also, notice that $B_4\subset B_{6/r_0}(p_0/r_0)$. Indeed, if $x\in B_4$ then 
$$ \left|x-\frac{p_0}{r_0}\right|\le |x|+\frac{|p_0|}{r_0}\le 4+\frac{2}{r_0}\le \frac{4}{r_0}
+ \frac{2}{p_0}=\frac{6}{r_0},$$
and so $x\in B_{6/r_0}(p_0/r_0)$. Therefore, we can say that $(u_{r_0},E_{r_0})$ 
is a minimizing pair in $B_4$. Moreover, by construction, $0\in \partial E_{r_0}$. 
Furthermore, we see that
$$ \int_{\R^n} \frac{|u_{r_0}(y)|}{1+|y|^{n+2s}}\,dy = 
r_0^{s+\frac{\sigma}{2}}\int_{\R^n} \frac{|u(y)|}{r_0^{n+2s}+|y|^{n+2s}}\,dy \le \Lambda_{r_0}, $$
for some $\Lambda_{r_0}>0$ depending on $r_0$. 
So we can apply Corollary \ref{67dfffsjjsj} to $(u_{r_0},E_{r_0})$ obtaining that 
\begin{equation}\label{67uio123}
\| u_{r_0}\|_{L^\infty(B_{2})}\le K, \end{equation}
for some $K$ that depends only on $n$, $s$ and $\sigma$. 

Now we set $\omega:=\frac{p_0-x_0}{r_0}$, and we observe that $\omega\in B_2$ 
(and so $-\omega\in B_2$). 
Therefore, from this and \eqref{67uio123} we get 
$$ |u_{r_0}(-\omega)|\le K.$$
On the other hand, recalling \eqref{rescaled-BIS}, 
$$ u_{r_0}(-\omega)=r_0^{\frac{\sigma}{2}-s} u(-r_0\omega +p_0)=r_0^{\frac{\sigma}{2}-s} u(x_0).$$ 
The last two formulas imply that 
$$ |u(x_0)|\le K r_0^{s-\frac{\sigma}{2}} =K (d(x_0))^{s-\frac{\sigma}{2}}, $$
which shows \eqref{CD-10-XX-BIS}. 

Finally, \eqref{CD-11-XX-BIS} follows from \eqref{CD-10-XX-BIS} and 
the fact that $0\in\partial E$. This concludes the proof of Corollary \ref{67dfffsjjsj-BIS}. 
\end{proof}

\subsection{Completion of the proof of Theorem~\ref{DEC}}

In order to end the proof of Theorem~\ref{DEC},
we recall a H\"older continuity property for nonlocal solutions:

\begin{lemma}\label{APPR:L}
Assume that~$(-\Delta)^s u=0$ in~$B_1$, with~$u\in L^\infty(B_1)$.
Then $u\in C^1(B_{1/2})$ and
$$ \|u\|_{C^1(B_{1/2})}\le C\left(
\|u\|_{L^\infty(B_1)}+
\int_{\R^n} \frac{|u(x)|}{1+|x|^{n+2s}}\,dx
\right),$$
for some~$C>0$.
\end{lemma}

\begin{proof}
First of all, by Theorem~2.6 of~\cite{APPR},
we know that~$u\in C^\alpha(B_{9/10})$.
Then we can apply Theorem~2.7 of~\cite{APPR}
(say, in~$B_{3/4}$) and obtain the desired result.
\end{proof}

Now we provide a rescaled version of Lemma \ref{APPR:L}. 

\begin{corollary}\label{coro:rescaled}
Assume that~$(-\Delta)^s u=0$ in~$B_t(q)$, for some $t>0$ and $q\in\R^n$, with~$u\in L^\infty(B_t(q))$.
Then $u\in C^1(B_{t/2}(q))$ and
$$ \|\nabla u\|_{L^\infty(B_{t/2}(q))}\le C\, t^{-1}\left(
\|u\|_{L^\infty(B_t(q))}+ t^{2s}\int_{\R^n\setminus B_t(q)} \frac{|u(x)|}{|x-q|^{n+2s}}\,dx
\right),$$
for some~$C>0$.
\end{corollary}

\begin{proof}
For any $x\in B_1$, we define $v(x):=u(tx+q)$. Notice that, by construction, 
$(-\Delta)^s v=0$ in~$B_1$, and $v\in L^\infty(B_1)$. Hence, we are in position to apply Lemma \ref{APPR:L} 
to the function $v$, obtaining that $v\in C^1(B_{1/2})$ and 
\begin{equation}\label{alfa0}
\|\nabla v\|_{L^\infty(B_{1/2})}\le C\left( \|v\|_{L^\infty(B_1)}+ 
\int_{\R^n} \frac{|v(x)|}{1+|x|^{n+2s}}\,dx
\right),\end{equation}
for some~$C>0$. Now we observe that 
\begin{equation}\label{alfa4}
\|v\|_{L^\infty(B_1)}=\|u\|_{L^\infty(B_t(q))} \quad {\mbox{ and }} \quad
\|\nabla v\|_{L^\infty(B_{1/2})}= t\|\nabla u\|_{L^\infty(B_{t/2}(q))}.
\end{equation}
Moreover, using the change of variable $y=tx+q$, 
\begin{equation}\begin{split}\label{alfa}
\int_{\R^n} \frac{|v(x)|}{1+|x|^{n+2s}}\,dx &=\, \int_{\R^n} \frac{|u(tx+q)|}{1+|x|^{n+2s}}\,dx \\
&=\, t^{2s}\int_{\R^n} \frac{|u(y)|}{t^{n+2s}+|y-q|^{n+2s}}\,dy. 
\end{split}\end{equation}
We observe that 
\begin{equation}\begin{split}\label{alfa1}
& t^{2s}\int_{B_t(q)} \frac{|u(y)|}{t^{n+2s}+|y-q|^{n+2s}}\,dy 
\le t^{2s}\|u\|_{L^\infty(B_t(q))}\int_{B_t(q)} \frac{dy}{t^{n+2s}}\\
& \qquad \le C t^{2s} \|u\|_{L^\infty(B_t(q))}
t^{n} t^{-n-2s} \le C\|u\|_{L^\infty(B_t(q))},
\end{split}\end{equation}
up to renaming $C$. Also,  
$$
\int_{\R^n\setminus B_t(q)} \frac{|u(y)|}{t^{n+2s}+|y-q|^{n+2s}}\,dy \le 
\int_{\R^n\setminus B_t(q)} \frac{|u(y)|}{|y-q|^{n+2s}}\,dy.  
$$
Using this and \eqref{alfa1} into \eqref{alfa} we obtain that 
$$ \int_{\R^n} \frac{|v(x)|}{1+|x|^{n+2s}}\,dx\le C
\|u\|_{L^\infty(B_t(q))} + t^{2s}\int_{\R^n\setminus B_t(q)} \frac{|u(y)|}{|y-q|^{n+2s}}\,dy.$$ 
Plugging this and \eqref{alfa4} into \eqref{alfa0} we get the 
desired result. 
\end{proof}

Now we can complete the proof of Theorem~\ref{DEC}.

\begin{proof}[Proof of Theorem~\ref{DEC}]
We define, for any~$r>0$,
\begin{equation}\label{rescaled}
u_r(x):=r^{\frac{\sigma}{2}-s} u(rx)\;{\mbox{ and }}\;
E_r:= \frac{1}{r}E
\end{equation}
and we apply Corollary~\ref{67dfffsjjsj}
to the minimizing pair~$(u_r,E_r)$, with~$r\in(0,1/2]$. 
For this, notice that 
$$ \int_{\R^n} \frac{|u_{r}(y)|}{1+|y|^{n+2s}}\,dy = 
r^{s+\frac{\sigma}{2}}\int_{\R^n} \frac{|u(y)|}{r^{n+2s}+|y|^{n+2s}}\,dy \le \Lambda_{r}, $$
for some $\Lambda_{r}>0$ depending on $r$. 
Then, \eqref{CD-2} follows from~\eqref{CD-3-XX}. 
Also,~\eqref{CD-2-BIS} is a consequence of~\eqref{CD-4-XX}.

Now we prove~\eqref{CD-1}.
For this, 
since Theorem~\ref{DEC} deals with interior estimates,
we may suppose that
\begin{equation}\label{688}
{\mbox{the minimizing property of~$(u,E)$
holds in~$B_{72}$ instead of~$B_1$.}}\end{equation}
Now we assume that~$s>\sigma/2$ and we fix~$x$, $y\in B_{1/2}$.
We claim that
\begin{equation}\label{tyuiodd}
|u(x)-u(y)|\le C\,|x-y|^{s-\frac\sigma2}.\end{equation}
Let~$p:=(x+y)/2$ and~$r:=|x-y|$.
Notice that we may suppose that
\begin{equation}\label{689}
r\le \frac{1}{100},\end{equation}
otherwise the fact that~$|u(x)-u(y)|\le 2\|u\|_{L^\infty(B_{1/2})}$ would give~\eqref{tyuiodd}.
Then, there are three possibilities:
\begin{eqnarray}
\label{PP-1} && B_{5r}(p)\setminus E =\varnothing,\\
\label{PP-2} && B_{5r}(p)\setminus E\ne \varnothing \ {\mbox{ and }}
u(x)=u(y)=0,
\\
\label{PP-3}&& B_{5r}(p)\setminus E\ne \varnothing \ {\mbox{
and either~$u(x)>0$
or~$u(y)>0$.}}
\end{eqnarray}
If~\eqref{PP-2} holds true then~\eqref{tyuiodd} is obvious,
therefore we consider only the possibilities~\eqref{PP-1}
and~\eqref{PP-3}.

If~\eqref{PP-1} holds, we consider $R>0$ such that $d(p)=R$, 
where we recall that $d(x)=\dist(x,\partial E)$ 
denotes the distance of $x$ from $\partial E$. By \eqref{PP-1}, we have that $5r<R<2$. 

Also, we have that 
\begin{equation}\label{wer21}
B_{10}(p)\subset B_{12},\end{equation} 
indeed if $x\in B_{10}(p)$ then 
$$ |x|\le |x-p|+|p|< 10+2=12,$$
and so $x\in B_{12}$. This proves \eqref{wer21}. 
Therefore, recalling \eqref{688} and using Corollary \ref{67dfffsjjsj}, 
we see that 
\begin{equation}\label{po10}
|u(x)|\le K (d(x))^{s-\frac{\sigma}{2}}, \quad {\mbox{ for any }} x\in B_{10}(p). 
\end{equation}
Also, if $x\in B_R(p)$, 
$$ d(x)\le |x-p| + d(p) \le R +R =2R, $$
therefore, from \eqref{po10} (recall that $R<2$), we obtain that 
\begin{equation}\label{po11}
|u(x)|\le K R^{s-\frac{\sigma}{2}}, \quad {\mbox{ for any }} x\in B_R(p), 
\end{equation}
up to renaming $K$. 

Now we use Lemma~2.3 in~\cite{DV}
and we obtain that~$(-\Delta)^s u=0$ in~$B_{R/2}(p)$. 
Moreover, from Corollary \ref{67dfffsjjsj} and recalling \eqref{688} and \eqref{wer21}, we have that 
$\|u\|_{L^\infty(B_{R/2}(p))}\le K$. So we are in position to apply Corollary \ref{coro:rescaled} 
(with $t:=R/2$ and $q:=p$), 
thus obtaining that
\begin{equation}\label{piou7680p-2}
\|\nabla u\|_{L^\infty(B_{R/4}(p))}\le C\, R^{-1}\left(
\|u\|_{L^\infty(B_{R/2}(p))}+ R^{2s}\int_{\R^n\setminus B_{R/2}(p)} \frac{|u(x)|}{|x-p|^{n+2s}}\,dx
\right).\end{equation}
We notice that, by \eqref{po11}, 
\begin{equation}\label{piou7680p}
\|u\|_{L^\infty(B_{R/2}(p))} \le K R^{s-\frac{\sigma}{2}}. 
\end{equation}
Moreover, 
\begin{equation}\begin{split}\label{qwer5432-2}
& \int_{\R^n\setminus B_{R/2}(p)} \frac{|u(x)|}{|x-p|^{n+2s}}\,dx \\=\,& 
\int_{\R^n\setminus B_{10}(p)} \frac{|u(x)|}{|x-p|^{n+2s}}\,dx
+ \int_{B_{10}(p)\setminus B_{R/2}(p)} \frac{|u(x)|}{|x-p|^{n+2s}}\,dx. \end{split}\end{equation}
Now we observe that, if $x\in \R^n\setminus B_{10}(p)$, then $10<|x-p|\le|x|+|p|\le|x|+2$, 
and so $|x|>8$. Hence, 
$$ |x-p|\ge |x|-2 =\frac34 |x| +\frac14|x| -2 \ge \frac34 |x| +2-2=\frac34 |x|. $$
Therefore, 
\begin{equation}\label{qwer5432}
\int_{\R^n\setminus B_{10}(p)} \frac{|u(x)|}{|x-p|^{n+2s}}\,dx
\le 
C\,\int_{\R^n\setminus B_{10}(p)} \frac{|u(x)|}{|x|^{n+2s}}\,dx\le C\,\Lambda. \end{equation} 
Furthermore, using \eqref{po10}, we obtain that, if $x\in B_{10}(p)\setminus B_{R/2}(p)$, then 
$$ |u(x)|\le K (d(x))^{s-\frac{\sigma}{2}} \le K\left(|x-p| + d(p)\right)^{s-\frac{\sigma}{2}}
= K\left(|x-p| + R\right)^{s-\frac{\sigma}{2}}.$$
As a consequence, 
$$
\int_{B_{10}(p)\setminus B_{R/2}(p)} \frac{|u(x)|}{|x-p|^{n+2s}}\,dx \le 
K\,\int_{B_{10}(p)\setminus B_{R/2}(p)}
\frac{\left(|x-p| + R\right)^{s-\frac{\sigma}{2}}}{|x-p|^{n+2s}}\,dx.
$$
So, by making the change of variable $y=(x-p)/R$, we obtain
\begin{eqnarray*}
&& \int_{ B_{10}(p)\setminus B_{R/2}(p) } \frac{|u(x)|}{|x-p|^{n+2s}}\,dx \le 
K\,\int_{B_{10/R}\setminus B_{1/2}}  
\frac{\left(R|y| + R\right)^{s-\frac{\sigma}{2}}}{R^{n+2s}|y|^{n+2s}}R^n\,dy \\
&&\qquad \le K\, R^{s- \frac{\sigma}{2} -2s } \int_{ \R^n \setminus B_{1/2} }
\frac{ (|y|+1)^{ s-\frac{\sigma}{2} } }{ |y|^{n+2s} }\,dy 
\le CK \, R^{ s- \frac{\sigma}{2} - 2s },
\end{eqnarray*}
for some $C>0$. This information, \eqref{qwer5432} and \eqref{qwer5432-2} give 
$$ \int_{\R^n\setminus B_{R/2}(p)} \frac{|u(x)|}{|x-p|^{n+2s}}\,dx \le 
C\,\Lambda + CK \, R^{ s- \frac{\sigma}{2} - 2s }.$$ 
Plugging this and \eqref{piou7680p} into \eqref{piou7680p-2}, we conclude that
\begin{equation}\label{234qweyuo}
\|\nabla u\|_{L^\infty(B_{R/4}(p))}\le C\, R^{-1}\left(
K R^{s-\frac{\sigma}{2}} + 
C\,\Lambda\, R^{2s} + CK \, R^{ s- \frac{\sigma}{2}}\right) 
\le C\, R^{s-\frac{\sigma}{2}-1},\end{equation}
up to relabelling $C$ (recall that $R<2$). 

From \eqref{234qweyuo} we obtain that, for any $a$, $b\in B_{R/4}(p)$,  
$$ |u(a)-u(b)|\le C \, R^{s-\frac{\sigma}{2}-1} |a-b|.$$
Since $R>5r$, we have that $x$, $y\in B_{R/4}(p)$, and so 
\begin{eqnarray*}
&& |u(x)-u(y)|\le C \, R^{s-\frac{\sigma}{2}-1} |x-y| 
\\&&\qquad = C \, R^{s-\frac{\sigma}{2}-1} r^{1-s+\frac{\sigma}{2}}|x-y|^{s-\frac{\sigma}{2}} 
\le C|x-y|^{s-\frac{\sigma}{2}},\end{eqnarray*}
where $C$ may change from step to step. 
This says that~\eqref{tyuiodd}
holds true in this case. 

Now let us suppose that~\eqref{PP-3} holds true.
Then there exist~$z\in B_{5r}(p)\setminus E$
and~$\eta\in\{x,y\}$ such that~$u(\eta)>0$.
In particular~$\eta\in E$ and so there exists~$\zeta$ on the segment
joining~$\eta$ and~$z$ such that~$\zeta\in\partial E$. Notice that,
since the ball is convex, we have that~$\zeta\in B_{5r}(p)$.

Hence, we have the following picture:
$\zeta\in \partial E$, $x$ and~$y$ lie in~$B_{3r}(p)$
and~$B_{1}(\zeta)\subseteq B_{2}$ (where the minimization property
holds, recall~\eqref{688} and~\eqref{689}).

Thus, with a slight abuse of notation, we suppose, up to
a translation, that~$\zeta=0$. So our picture
becomes that~$0\in \partial E$, $x$ and~$y$ lie in~$B_{10r}$,
with our minimizing property in~$B_1$.

So we consider the minimizing pair~$(u_r,E_r)$
as in~\eqref{rescaled}, which is minimizing in~$B_{1/r}\supseteq B_{50}$
(recall~\eqref{689}).
In this way,
we apply formula~\eqref{CD-4-XX}, thus obtaining
$$ \| u_r\|_{L^\infty(B_{25})}\le K .$$ 
Notice that~$r^{-1} x$, $r^{-1} y\in B_{10}\subset B_{25}$, hence
$$ |u_r(r^{-1}x)|+|u_r(r^{-1}y)|\le 2K.$$
So we obtain
\begin{eqnarray*}
&& |u(x)-u(y)| = r^{s-\frac\sigma2} |u_r(r^{-1}x)- u_r(r^{-1}y)|
\\ &&\qquad\le 2K r^{s-\frac\sigma2}=2K |x-y|^{s-\frac\sigma2}.
\end{eqnarray*}
This proves~\eqref{tyuiodd}, which in turn implies~\eqref{CD-1},
thus completing the proof of
Theorem~\ref{DEC}.
\end{proof}

We complete this paper with a brief comment about the two-phase case
(i.e. when the function~$u$ in Theorem~\ref{DEC}
is not assumed to be nonnegative to start with).
The additional difficulties in this setting arise
since the fractional harmonic replacements do not behave nicely
with respect to the operation of taking the positive part,
namely the positive part of the harmonic replacement
is not necessarily
harmonic in its positive set. As an example,
considering the fractional harmonic replacement
introduced in \cite{DV},
one can consider the fractional harmonic function~$u(x):=x_+^s-1$
in~$(0,+\infty)$,
with fixed boundary data in~$(-\infty,0)\cup(1,+\infty)$;
similarly, in the case of the fractional harmonic replacement
introduced here in Section~\ref{FHR},
one can consider the case~$s=1/2$ and the harmonic function
on~$\R^2$ given by~$u(x,y)=xy$.

This difficulty arising at the level of the fractional replacements
in the two-phase problem
reflects also into the proof of the density estimates here
(precisely in the computations below \eqref{fe7d}
and~\eqref{89UhgfdFG}).

For these reasons, we believe that the investigation
of density estimates and continuity properties for two-phase fractional
minimizers is an interesting open problem.

\end{document}